\documentclass[notitlepage,reqno,11pt]{amsart}
\usepackage{latexsym,amssymb, epsfig, amsmath,amsfonts, subfigure,amsthm}

\usepackage{rotating}
\usepackage[toc,page]{appendix}
\usepackage{color}
\usepackage{multirow}
\usepackage{relsize}
\usepackage{microtype}
\usepackage[foot]{amsaddr}

\usepackage[colorlinks, allcolors=blue]{hyperref}

\usepackage[margin=1in]{geometry}

\numberwithin{equation}{section}
 \newtheorem{assumption}{Assumption}[section]
\newtheorem{lemma}{Lemma}[section]
\newtheorem{theorem}{Theorem}[section]

\newtheorem{coro}{Corollary}[section]
\newtheorem{prop}{Proposition}[section]
\newtheorem{remark}{Remark}[section]

\newlength{\defbaselineskip}
\newcommand{\setlinespacing}[1]%
           {\setlength{\baselineskip}{#1 \defbaselineskip}}

\newcommand{\RR}{{\mathbb R}}

\newcommand{\NN}{{\mathbb N}}

\def\E{\mathbb{E}}
\def\P{\mathbb{P}}
\def\R{\mathbb{R}}

\newcommand{\sF}{{\mathcal{F}}}
\newcommand{\sI}{{\mathcal{I}}}

\newcommand{\beql}[1]{\begin{equation}\label{#1}}
\newcommand{\eeq}{\end{equation}}

\newcommand{\beqal}[1]{\begin{eqnarray}\label{#1}}
\newcommand{\eeqa}{\end{eqnarray}}
\newcommand{\beq}{\begin{displaymath}}
\newcommand{\eeqno}{\end{displaymath}}
\newcommand{\bali}[1]{\begin{align}\label{#1}}
\newcommand{\eali}{\begin{align}}
\newcommand{\balino}{\begin{align*}}
\newcommand{\ealino}{\begin{align*}}

\newcommand{\ep}{\epsilon}

\newcommand{\Var}{\text{\rm Var}}

\newcommand{\wt}{\widetilde}

\newcommand{\mfI}{\mathfrak{I}}
\newcommand{\mfi}{\mathfrak{i}}

\newcommand{\bone}{{\mathbf 1}}

\newcommand{\qandq}{\quad\mbox{and}\quad}

\newcommand{\qforq}{\quad\mbox{for}\quad}

\newcommand{\qasq}{\quad\mbox{as}\quad}
\newcommand{\qinq}{\quad\mbox{in}\quad}

\newcommand{\non}{\nonumber}
\newcommand{\RA}{\Rightarrow}

\newcommand{\baa}{\begin{eqnarray*}}
\newcommand{\eaa}{\end{eqnarray*}}

\newcommand{\ttl}{\Large Functional law of large numbers and PDEs for epidemic models \\[5pt] with infection-age dependent infectivity}

\begin{document}

\title[]{\ttl}

\author[Guodong \ Pang]{Guodong Pang$^*$}
\address{$^*$Department of Computational Applied Mathematics and Operations Research,
George R. Brown College of Engineering,
Rice University,
Houston, TX 77005}
\email{gdpang@rice.edu}

\author[{\'E}tienne \ Pardoux]{{\'E}tienne Pardoux$^\dag$}
\address{$^\dag$Aix--Marseille Univ, CNRS, I2M, Marseille, France}
\email{etienne.pardoux@univ-amu.fr}

%\date{\today}
%
%

\begin{abstract} 
We study epidemic models where the infectivity of each individual is a random function of the infection age (the elapsed time of infection). 
To describe the epidemic evolution dynamics, we use a stochastic process that tracks the number of individuals at each time that have been infected for less than or equal to a certain amount of time, together with the aggregate infectivity process. 
We establish the functional law of large numbers (FLLN) for the stochastic processes that describe the epidemic dynamics. The limits are described by a set of deterministic Volterra-type integral equations, which has a further characterization using PDEs under some regularity conditions.  
The solutions are characterized with boundary conditions that are given by a system of Volterra equations. 
We also characterize the equilibrium points for the PDEs in the SIS model with infection-age dependent infectivity. 
To establish the FLLNs, we employ a useful criterion for weak convergence for the two-parameter processes together with useful representations for the relevant processes via Poisson random measures. 
\end{abstract}

\keywords{Functional law of large numbers, deterministic Volterra integral equations, PDEs, non--Markovian epidemic models, infection-age dependent (varying) infectivity, Poisson random measure, SIR, SIS, equilibrium in the SIS model}

\maketitle

\allowdisplaybreaks

\section{Introduction}

Kermack and McKendrick pioneered the introduction of PDE models to describe the epidemic dynamics for models with infection-age dependent (variable) infectivity in 1932 \cite{KM32}.  The underlying assumption of their model is that the infectious periods have a general distribution with density which is modeled through an infection--age dependent recovery rate, the infectious individuals having an infection--age dependent  infectivity, and the recovered ones a recovery--age susceptibility. In the present paper, we do not consider possible loss of immunity. We defer to a work in preparation the study of variable susceptibility. In the present paper, we mainly consider the SIR model (although we can allow for an exposed period, as will be explained below) and the SIS model. This work is a continuation of 
our first work on non--Markov epidemic models \cite{PP-2020}, and our work on varying infectivity models
\cite{FPP2020b}, see also \cite{PP2020-FCLT-VI}. In those papers, we show that certain deterministic Volterra type integral equations are Functional Law of Large Numbers (FLLN) limits of adequate individual based stochastic models. An important feature of our stochastic models is that they are non--Markov (since the infectious duration need not have an exponential distribution), and as a result the limiting deterministic models are equations with memory. Note that as early as in 1927, Kermack and McKendrick introduced in their seminal paper \cite{KM27} a SIR model with both infection--age dependent infectivity and infection--age dependent recovery rate, the latter allowing the infectious period to have an arbitrary absolutely continuous distribution (the infection--age dependent recovery rate is the hazard rate function of the infectious period). One part of that paper is devoted to the simpler case of constant rates, and apparently most of the later literature on epidemic models has concentrated on this special case, which leads to simpler ODE models, the corresponding stochastic models being Markov models, at the price of the models being less realistic. For example, the recent studies in Covid-19 \cite{FPP2020a,FodorKatzKovacs} indicates that using the ODE models can lead to an underestimation of the basic reproduction number $R_0$. 

In this paper, we go back to the original model of  Kermack and McKendrick \cite{KM27}, with two new aspects. First, as in our previous publications, we want to obtain the deterministic model as a law of large numbers limit of stochastic models, and second, we distribute the various infected individuals at time $t$ according to their infection--age, and establish a PDE for the ``density of individuals'' being infected at time $t$, with infection--age $x$.

In our stochastic epidemic model, each individual is associated with a random infectivity, which varies as a function of the age of infection (elapsed time since infection). 
The random infectivity functions, effective during the infected period, are assumed to be i.i.d. for the various individuals, and will also generate the infectious period. The infectivity function is assumed to be  c\`adl\`ag  with a given number of discontinuities, and upper bounded by a deterministic constant. In particular, the law of the infectious period can be completely arbitrary. Our modeling approach allows the random infectivity functions to have an initial period of time during which they take zero values, corresponding to the exposed period. Thus our model generalizes both the classical SIR and SEIR models.  
 To describe the epidemic dynamics of the model, we use a (two-parameter or measure-valued) stochastic process that tracks at each time $t$ the number of individuals that have been infected for a duration less than or equal to a certain amount of time $x$, and an associated aggregate infectivity process which at each time $t$  sums up the infectivities of all individuals who are infected. 
From these processes, we can describe the cumulative infection process, the total number of infected individuals as well as the number of recovered ones at each time. 
We use similar processes to describe the epidemic dynamics for the SIS model with infection-age dependent infectivity. 

In the asymptotic regime of a large population (i.e., as the total population size $N$ tends to infinity), we establish the FLLN for the epidemic dynamics. The limits are characterized by  a set of deterministic Volterra-type integral equations (Theorem \ref{thm-FLLN}). 
Under certain regularity conditions,  the density function of the two-parameter (calendar time and infection age) limit process  can be described by a one-dimensional PDE 
 (Proposition \ref{prop-PDE}  in the case where the distribution of the infectious period is absolutely continuous). 
 Its solution is characterized with a boundary condition satisfying a one-dimensional Volterra-type integral equation. The aggregate infectivity limit process can be described by an integral of the average infectivity function with respect to the limiting two-parameter infectious process (Corollary \ref{coroFormula}, see also Remark \ref{rem-lambda-indep}). 
For the classical SIR model, we recover the well-known linear PDE first proposed by  Kermack and McKendrick \cite{KM32}. 
We further derive the PDE model when the distribution of the infectious period 
need not be absolutely continuous (Proposition \ref{prop-PDE-measure} and see also Corollary \ref{prop-PDE-det} where  the infectious periods are deterministic). These PDE models are new to the literature of epidemiology. 
For the SIS model, we also describe the limiting epidemic dynamics and the PDE representations, and derive the equilibrium quantities associated with the PDE and total count limit (assuming convergence to the equilibria).

\subsection{Literature review}

Non--Markov stochastic epidemic models lead (via the FLLN) to deterministic models, which are either low dimensional evolution equation with memory (i.e., Volterra type integral equations), or  else coupled ODE/PDE models, where the two variables are the time and the age of infection (time since infection). The first paper of Kermack and McKendrick \cite{KM27} adopts the first point of view, and the two next \cite{KM32,KM33} the second one. In our recent previous work on this topic \cite{PP-2020,FPP2020b}, we have adopted the first description. The goal of the present paper is to show that in the limit of a large population, our stochastic individual based model with age of infection dependent infectivity and recovery rate converges as well to a limiting system of PDE/ODEs. 

While the general model from \cite{KM27} was largely neglected until rather recently, most of the literature concentrating on the particular case of constant rates, there has been since the 1970s some papers considering infection-age dependent epidemic models, see in particular \cite{hoppensteadt1974age}. More recently, several papers have introduced coupled PDE/ODE models for studying age of infection dependent both infectivity and recovery rate, see in particular \cite{thieme1993may,inaba2004mathematical,ZhangPeng2007,magal2013two,YChen2014} and Chapter 13 in \cite{martcheva2015introduction}. 
In  \cite{clemenccon2008stochastic}, the authors consider a stochastic epidemic model with contract-tracing, tracking the infection duration since detection for each individual, and use a measure-valued Markov process to describe the epidemic dynamics. They prove a FLLN with a large population and establish a PDE limit, and also prove a FCLT with a SPDE limit process.
Since the beginning of the Covid--19 pandemic, a huge number of papers have been produced, with various models of the propagation of this disease. Most of them use ODE models, but a few, notably \cite{Kaplan2020,Gaubert,foutel2020individual,duchamps2021general} consider age of infection dependent infectivity, and possibly recovery rate. 
The last two 
derive the ODE/PDE model as a law of large numbers limit of stochastic individual based models. 
The article \cite{foutel2020individual} considers a branching process approximation of the early phase of an epidemic, and the way they model the dependence of the rate of infection with respect to the age of infection is less general than in our model. Recently, the authors in \cite{duchamps2021general} study contact tracing in an individual-based epidemic model via an ``infection graph" of the population, and prove the local convergence of the random graph to a Poisson marked tree and a  Kermack and McKendrick type of PDE limit for the dynamics by tracking the infection age.

Note also that one way that many authors have chosen in order to improve the realism of ODE models is by increasing the number of compartments. For instance, dividing the infectious compartment into subcompartments, each one corresponding to a different infection rate, is a way to introduce a (piecewise constant) infection age dependent infectivity. In a way, this means approaching a non-Markov process of a given dimension by a higher dimensional Markov process, or approaching a system differential equations with memory by a higher dimensional system of ODEs. 
In the present paper, we show that the system of integral equations with memory introduced in our earlier work 
\cite{PP-2020} can be replaced by an ODE/PDE system, i.e., an infinite dimensional differential equation. At the level of the stochastic finite population model, this means replacing a non-Markov finite dimensional Markov process by a high dimensional 
process (whose dimension is bounded by the total population size $N$, which tends to infinity in our asymptotic).
See Remarks \ref{infinite-dim} and \ref{infinite-dim-m} below.

We also like to mention the relevant work in queueing systems where the elapsed service times are tracked using two-parameter or measure-valued processes. The most relevant to us are the infinite-server (IS) queueing models studied in \cite{pang2010two,pang2017two,aras2017heavy}, where FLLN and FCLT are established for two-parameter processes to tracking elapsed and residual service times. However, the proof techniques we employ in this paper are very different from those papers. Here we exploit the representations with Poisson random measures and use a new weak convergence criterion (Theorem \ref{thm-DD-conv0}). In addition, despite similarities with the IS queueing models, the stochastic epidemic models
have an arrival (infection) process that depend on the state of the system. As a consequence, the limits in the FLLNs result in PDEs while the IS queueing models do not.

\subsection{Organization of the paper}
The paper is organized as follows. 
In Section \ref{sec-SIR}, we describe the stochastic epidemic model with infection-age dependent infectivity, and state the FLLN. 
In Section \ref{sec-PDE}, we present the PDE models from the FLLN limits, and also characterize the solution properties of the PDEs. 
The limits and PDE for the SIS model are presented in Section \ref{sec-SIS}, which also  considers the equilibrium behavior. 
In Sections \ref{sec-proof-SIR}, we prove the FLLN.  The Appendix gives the proof of the convergence criterion in Theorem \ref{thm-DD-conv0}. 

\subsection{Notation}
 All random variables and processes are defined on a common complete probability space $(\Omega, \sF, \P)$. The notation $\RA$ means convergence in distribution. We use $\bone_{\{\cdot\}}$ for the indicator function, and occasionally use $\bone\{\cdot\}$ for better readability. 
Throughout the paper, $\NN$ denotes the set of natural numbers, and $\RR^k (\RR^k_+)$ denotes the space of $k$-dimensional vectors
with  real (nonnegative) coordinates, with $\RR (\RR_+)$ for $k=1$.  For $x,y \in\RR$, we denote $x\wedge y = \min\{x,y\}$ and $x\vee y = \max\{x,y\}$. 
Let $D=D(\RR_+;\RR)$ denote the space of $\RR$--valued c{\`a}dl{\`a}g functions defined on $\RR_+$. 
Throughout the paper, convergence in $D$ means convergence in the  Skorohod $J_1$ topology, see Chapter 3 of \cite{billingsley1999convergence}. 
 Also, $D^k$ stands for the $k$-fold product equipped with the product topology. 
 Let $C$ be the subset of $D$ consisting of continuous functions.   Let $C^1$ consist of all differentiable functions whose derivative is continuous. Let $D_\uparrow$ denote the set of increasing functions in $D$. 
 Let $D_D= D(\RR_+; D(\RR_+;\RR))$ be the $D$-valued $D$ space, and the convergence in the space $D_D$ means that both $D$ spaces are endowed with the Skorohod $J_1$ topology.  The space $C_C$ is equivalent to $C(\RR_+^2; \RR_+)$. Let $C_\uparrow(\R_+^2;\R_+)$ denote the space of continuous functions from $\R_+^2$ into $\R_+$, which are increasing as a function of their second variable. 
For any increasing c{\`a}dl{\`a}g function $F(\cdot): \R_+\to \R_+$, abusing notation, we write $F(dx)$ by treating $F(\cdot)$ as the positive (finite) measure on 
$\R_+$ whose distribution function is $F$. 
For any $\RR$--valued c{\`a}dl{\`a}g function $\phi(\cdot)$ on $\R_+$, the integral $\int_{a}^b \phi(x)F(dx)$ represents $\int_{(a,b]} \phi(x) F(dx)$ for $a<b$.

\section{Model and FLLN} \label{sec-SIR}

\subsection{Model description}
We consider an epidemic model
  in which the infectivity rate depends on the age of infection (that is, how long the individuals have been infected).  
  Specifically,
 each  individual $i$ is associated with an infectivity process $\lambda_i(\cdot)$, and we assume that these random functions  are i.i.d.. Let $\eta_i= \inf\{t>0: \lambda_i(r) = 0, \, \forall r \ge t\}$ be the infected period corresponding to the individual that gets infected at time $\tau^N_i$. The $\eta_i$'s are i.i.d., with a cumulative distribution function (c.d.f.) $F$.  Let $F^c=1-F$.

Individuals are grouped into susceptible, infected and recovered ones. 
Let the population size be $N$ and 
$S^N(t)$, $I^N(t)$ and $R^N(t)$ denote the numbers of the susceptible, infected and recovered individuals at time $t$.   We have the balance equation:
$
N= S^N(t) + I^N(t) + R^N(t)$, $ t \ge 0. 
$
Assume that $S^N(0)>0$, $I^N(0)>0$ and $R^N(0)=0$. Let $\mfI^N(t,x)$ be the number of infected individuals at time $t$ that have been infected for a duration less than or equal to $x$. 
Note that for each $t$, $\mfI^N(t,x)$ is nondecreasing in $x$, which is the distribution of $I^N(t)$ over the infection-ages.  
Let $A^N(t)$ be the cumulative number of newly infected individuals in $(0,t]$, with the infection times $\{\tau^N_i: i \in \NN\}$.

Let $\{\tau_{j,0}^N, j =1,\dots, I^N(0)\}$ be the times at which the initially infected individuals at time 0 became infected. 
Then $\tilde{\tau}_{j,0}^N = -\tau_{j,0}^N$, $ j =1,\dots, I^N(0)$, represent the amount of time that an initially infected individual has been infected by time $0$, that is, the age of infection at time $0$. 
WLOG, assume that $0 > \tau_{1,0}^N> \tau_{2,0}^N> \cdots >\tau_{I^N(0),0}^N$ (or equivalently 
$0 < \tilde{\tau}_{1,0}^N< \tilde{\tau}_{2,0}^N< \cdots < \tilde{\tau}_{I^N(0),0}^N$). Set $\tilde{\tau}_{0,0}^N=0$. 
We define $\mathfrak{I}^N(0,x) = \max\{j \ge 0: \tilde{\tau}_{j,0}^N \le x\}$, the number of initially infected individuals that have been infected for a duration less than or equal to $x$ at time $0$. Assume that there exists $0 \le \bar{x}< \infty$  such that $I^N(0) = \mfI^{N}(0, \bar{x})$ a.s. 

Each initially infected individual $j=1,\dots, I^N(0)$, is associated with an infectivity process $\lambda_j^0(\cdot)$, and we assume that they are also i.i.d., with the same law as $\lambda_i(\cdot)$. This is reasonable since it is for the same disease, and the infectivity for the initially and newly infected individuals with the same infection age should have the same law.  The infectivity processes take effect at the epochs of infection. 
For each $j$, let  $\eta^0_j =\inf\{t>0: \lambda_j^0(\tilde{\tau}_{j,0}^N+r)=0,\,  \forall r \ge t\}$ be the remaining infectious period, which depends on the elapsed infection time $\tilde{\tau}_{j,0}^N$, but is independent of the elapsed infection times of other initially infected individuals.  
In particular, the conditional distribution of $\eta^0_j$ given that $\tilde{\tau}_{j,0}^N=s>0$ is given by
\begin{align} \label{enq-eta0-age}
\P(\eta^0_j > t | \tilde{\tau}_{j,0}^N=s) = \frac{F^c(t+s)}{F^c(s)}, \qforq t, s >0. 
\end{align}
Note that the $\eta^0_j$'s are independent but not identically distributed.

 For an initially infected individual $j=1,\dots,I^N(0)$, 
 the infection age is given by $\tilde{\tau}^N_{j,0}+ t$ for $0 \le t \le \eta^0_j$, during the remaining infectious period.
 For a newly infected individual $i$, the infection age is given by 
 $t- \tau^N_i$, for $ \tau^N_i \le t \le \tau^N_i +\eta_i$ during the infectious period. 
 Note that $\lambda_i(\cdot)$ and $\lambda^0_j(\cdot)$ are equal to zero on $\RR_{-}$. 
 
The aggregate infectivity process at time $t$ is given by 
 \begin{align} \label{eqn-cI-n}
\mathcal{I}^N(t) = \sum_{j=1}^{I^N(0)} \lambda_j^0 (\tilde{\tau}^N_{j,0}+t)   
+ \sum_{i=1}^{A^N(t)} \lambda_i(t-\tau^N_i) , \quad t \ge 0.
 \end{align}
 (Note that the notation $\mfI^N$ was used for the infectivity process in \cite{FPP2020b,PP2020-FCLT-VI}.)
The instantaneous infection rate at time $t$ can be written as 
  \begin{align}  \label{eqn-Upsilon}
 \Upsilon^N(t) & =  \frac{S^N(t)}{N} \mathcal{I}^N(t), \quad t \ge 0. 
 \end{align}
 The counting process of newly infected individuals $A^N(t)$ can be written as 
\begin{align} \label{eqn-An-PRM}
A^N(t) = \int_0^t \int_0^\infty \bone_{u \le \Upsilon^N(s^-) } Q(ds,du)\,, 
\end{align}
where $Q$ is a standard Poisson random measure on $\RR^2_+$  (see, e.g., \cite[Chapter VI]{Cinlar2011probability}).

Among the initially infected individuals, the number of individuals who have been infected for a duration less than or equal to $x$ at time $t$ is equal to  
\begin{align} \label{eqn-In0-rep}
\mfI^{N}_0(t,x) = \sum_{j=1}^{I^N(0)} \bone_{\eta_j^0 > t} \bone_{\tilde{\tau}_{j,0}^N  \le (x-t)^+ } = \sum_{j=1}^{\mfI^N(0, (x-t)^+)} \bone_{\eta_j^0 > t}\,, \quad t, x\ge 0,
\end{align}
Recall the age limit of the initially infected individuals $\bar{x}$ at time zero. 
Thus, the number of the initially infected individuals that remain infected at time $t$ can be written as
\begin{equation} \label{eqn-In0-total}
I^N_0(t) \;=\;\mfI^N_0(t,\bar{x}+t)\,, \quad t \ge 0\,. 
\end{equation}

Among the newly infected individuals, the number of individuals who have been infected for a duration  less than or equal to $x$ at time $t$ is equal to
\begin{align} \label{eqn-In1-rep}
\mfI^{N}_1(t,x) &=  \sum_{i=1}^{A^N(t)} \bone_{(t-x)^+<\tau^N_i \le t} \bone_{\tau^N_i + \eta_i >t} = \sum_{i=1}^{A^N(t)}  \bone_{\tau^N_i + \eta_i >t}- \sum_{i=1}^{A^N((t-x)^+)}  \bone_{\tau^N_i + \eta_i >t}\non \\
&=  \sum_{i=A^N((t-x)^+)+1}^{A^N(t)}  \bone_{\tau^N_i + \eta_i >t} 
\end{align}

Thus, the number of newly infected individuals that remain infected at time $t$ can be written as
\begin{equation}\label{eqn-In1-total}
I^N_1(t) = \mfI^N_1(t,t). 
\end{equation}

We also have the total number of individuals infected at time $t$ that have been infected for a duration which is less than or equal to $x$:
$$
\mfI^N(t,x) = \mfI^N_0(t,x) + \mfI^N_1(t,x), \quad t \ge 0, \, x \ge 0.
$$
Note that for each $t$,  the support of the measure $\mfI^N_0(t,dx)$ is included in $[0, t+\bar{x}]$ and 
the support of the measure $\mfI^N_1(t, dx)$ is included in $[0,t]$.  Thus
$$
I^N(t) = \mfI^N_0(t,t+\bar{x}) + \mfI^N_1(t,t) =\mfI^N(t,\infty), \quad t \ge 0. 
$$
Here we occasionally use $\infty$ in the second component for convenience with the understanding that  $\mfI^N_0(t,x) = \mfI^N_0(t,t+\bar{x})$ for $x>t+\bar{x}$ and  $\mfI^N_1(t,x) =\mfI^N_1(t,t) $ for $x>t$.

We also have for $t\ge 0$,
\begin{align}
S^N(t) &= S^N(0) - A^N(t),  \label{eqn-Sn-rep}\\
R^N(t) &= \sum_{j=1}^{I^N(0)} \bone_{\eta_j^0 \le t}  + \sum_{i=1}^{A^N(t)} \bone_{\tau^N_i + \eta_i \le t}\,.
  \label{eqn-Rn-rep}
\end{align}

We remark that the sample paths of $\mfI^N(t,x)$ belong to the space $D_D$, denoting $ D(\RR_+;D(\RR_+;\RR))$, the $D$-valued $D$ space,  but not in the space $D(\RR_+^2; \RR)$. We prove the weak convergence in the space $D_D$ where both $D$ spaces are endowed with the Skorohod $J_1$ topology. Note that the space $D(\RR_+^2; \RR)$ is a strict subspace of $D_D$, although they are equivalent in the continuous cases, that is, $C(\RR_+^2; \RR) = C_C$. 
See more discussions on these spaces in \cite{pang2010two,pang2017two, balan2019weak, bickel1971convergence}.

\begin{remark}{\bf The SEIR model}.
Suppose that $\lambda_i(t)=0$ for $t\in[0,\xi_i)$, where $\xi_i<\eta_i$, and denote $I$ as the compartment of infected (not necessarily infectious) individuals. An individual who gets infected at time $\tau^N_i$ is first exposed during the time interval $[\tau^N_i,\tau^N_i+\zeta_i)$, and then infectious during the time interval
$(\tau^N_i+\zeta_i,\tau^N_i+\eta_i)$. One may state that the individual is infected during the time interval $[\tau^N_i,\tau^N_i+\eta_i)$.
At time $\tau^N_i+\eta_i$, he recovers. All what follows covers perfectly this situation. In other words, our model accomodates perfectly an exposed period before the infectious period, which is important for many infectious diseases, including the Covid--19. However, we distinguish only three compartments, $S$ for susceptible, $I$ for infected (either exposed or infectious), $R$ for recovered.
\end{remark}

 In the sequel, the time interval $[\tau^N_i,\tau^N_i+\eta_i)$ will be called the infectious period, although it might rather be the period during which the individual is infected (either exposed or infectious). 

\subsection{FLLN} 

Define the LLN-scaled processes  $\bar{X}^N= N^{-1} X^N$ for any processes $X^N$. 
We make the following assumptions on the initial quantities. 

\begin{assumption} \label{AS-FLLN-Initial}
There exists a deterministic continuous nondecreasing function $\bar{\mfI}(0,x)$ for $x\ge 0$ with $\bar{\mfI}(0,0)=0$ such that 
$\bar{\mfI}^{N}(0,\cdot) \to \bar{\mfI}(0,\cdot)$ in $D$  in probability as $N\to\infty$. 
 Let $\bar{I}(0) = \bar{\mfI}(0, \bar{x})$.  Then $(\bar{I}^N(0), \bar{S}^N(0)) \to (\bar{I}(0), \bar{S}(0)) \in (0,1)^2$  in probability as $N\to \infty$ where  $ \bar{S}(0) =1 - \bar{I}(0)\in (0,1)$. 
\end{assumption}

\begin{remark} 
Recall that $\bar{\mfI}^{N}(0,\cdot)$ describes the distribution of the initially infected individuals over the ages of infection. The assumption means that there is a corresponding limiting continuous distribution as the population size goes to infinity.

Suppose now that the r.v.'s $\{\tau_{j,0}^N\}_{1\le j\le N}$ are not ordered, but rather i.i.d., with a common distribution function $G$ which we assume to be  continuous. It then follows from the law of large numbers that Assumption \ref{AS-FLLN-Initial} holds in this case.
\end{remark} 

We make the following assumption on the random function $\lambda$.

\begin{assumption} \label{AS-lambda}
  Let $\lambda(\cdot)$ be a process having the same law of $\{\lambda_j^0(\cdot)\}_j$ and $\{\lambda_i(\cdot)\}_i$. 
Assume that there exists a constant $\lambda^*$ such that for each $0<T<\infty$, 
$\sup_{t\in [0,T]} \lambda(t) \le \lambda^*$  almost surely.
Assume that there exist an integer $k$, a random sequence  $0=\zeta^0 \le \zeta^1 \le \cdots \le \zeta^k $ and associated random functions $\lambda^\ell \in C(\RR_+;[0,\lambda^\ast])$, $1\le\ell \le k$, such that 
\begin{align} \label{eqn-lambda-assump}
\lambda(t) = \sum_{\ell=1}^k \lambda^\ell(t) \bone_{[\zeta^{\ell-1},\zeta^\ell)}(t).
\end{align}
In addition, 
we assume that there exists a deterministic nondecreasing function $\varphi \in C(\RR_+;\RR_+)$ with $\varphi(0)=0$ such that $|\lambda^\ell(t) - \lambda^\ell(s)| \le \varphi(t-s)$ almost surely for all $t,s \ge 0$ and for all $\ell\ge 1$. 
Let $\bar{\lambda}(t) = \E[\lambda_i(t)] =\E[\lambda^0_j(t)]$
and $v(t) =\Var(\lambda(t)) = \E\big[\big(\lambda(t) - \bar\lambda(t)\big)^2\big]$  for $t\ge 0$.  
\end{assumption}

\begin{remark} \label{rem-R0}
Recall that the basic reproduction number $R_0$ is the mean number of susceptible individuals whom an infectious individual infects in a large population otherwise fully susceptible. In the present model, clearly
\[ R_0=\int_0^\infty \bar\lambda(t)dt\,.\]
 Suppose that $\lambda_i(t)=\tilde{\lambda}(t)\bone_{t<\eta_i}$, where $\tilde{\lambda}(t)$ is a deterministic function. Then
 \[ R_0=\int_0^\infty \tilde{\lambda}(t)F^c(t)dt\,.\]
In the standard SIR model with $\tilde{\lambda}(t) \equiv \lambda$ and $\E[\eta] = \int_0^\infty F^c(t)dt$, the formula above reduces to the well known $R_0 = \lambda \E[\eta]$. See, e.g., \cite{britton2018stochastic}. 
 We obtain the same formula if the deterministic function $\tilde{\lambda}(t)$ is replaced by a process $\lambda_i(t)$  independent of $\eta_i$, with mean $\tilde{\lambda}(t)$.
 More precisely, in that case the sequence $({\lambda}_i(t),\eta_i)_{i\ge1}$ is assumed to be i.i.d., and for each $i$, ${\lambda}_i$ and $\eta_i$ are independent.  
\end{remark}

 The proof of the following Theorem, which is the main result of this section, will be given in section \ref{sec-proof-SIR}. 
 For a function $u(t,x) \in D_{D_\uparrow}$, we use the equivalent notations $d_x u(t,x)$ and $u_x(t,x)$ for the partial derivative w.r.t. $x$, while $u(t,dx)$ denotes the measure whose distribution function is
 $x\mapsto u(t,x)$, which coincides with $u_x(t,x)dx$ if that last map is differentiable. In particular, $u_x(t,0)$ indicates the partial derivative evaluated at $x=0$.

\begin{theorem} \label{thm-FLLN}
Under Assumptions \ref{AS-FLLN-Initial} and \ref{AS-lambda}, as $N\to \infty$,
\begin{align}
\big(\bar{S}^N, \overline{\sI}^N, \bar{\mfI}^{N},   \bar{R}^N\big) \to\big(\bar{S},\overline{\sI}, \bar{\mfI}, \bar{R}\big) \ \text{in probability, locally uniformly in $t$ and $x$}, 
\end{align}
where the limits are the unique continuous solution to the following set of integral equations,  for $t, x\ge 0$, 
\begin{align}
\bar{S}(t) &= 1- \bar{I}(0) -  \int_0^t\bar{\Upsilon}(s) ds, \label{eqn-barS}\\
 \overline{\sI}(t) 
  &= \int_0^{\bar{x}} \bar\lambda(y+t)  \bar{\mfI}(0, dy)  + \int_0^t  \bar\lambda(t-s)  \bar\Upsilon(s) ds\,,
  \label{eqn-overline-cal-I-2}\\
\bar{\mfI}(t,x) &=   \int_0^{(x-t)^+} \frac{F^c(t+y)}{F^c(y)}  \bar{\mfI}(0, dy)   +  \int_{(t-x)^+}^t F^c(t-s) \bar{\Upsilon}(s) ds, \label{eqn-barI}\\
\bar{R}(t) &=\int_0^{\bar{x}} \left(1- \frac{F^c(t+y)}{F^c(y)} \right)  \bar{\mfI}(0, dy) + \int_0^t F(t-s) \bar{\Upsilon}(s) ds,  \label{eqn-barR}
\end{align}
with 
\begin{equation} \label{eqn-bar-Upsilon}
\bar{\Upsilon}(t) =\bar{S}(t)  \overline{\sI}(t) =   \bar\mfI_x(t,0)\,. 
\end{equation}
 The function $\bar{\mfI}(t,x)$ is nondecreasing in $x$ for each $t$.
 As a consequence, $\bar{I}^N\to \bar{I}$ in $D$ in probability as $N\to\infty$ where 
 \begin{align}\label{eqn-barIt}
 \bar{I}(t)= \bar{\mfI}(t, t+\bar{x}) =  \int_0^{\bar{x}} \frac{F^c(t+y)}{F^c(y)}  \bar{\mfI}(0, d y)   +  \int_{0}^t F^c(t-s) \bar{\Upsilon}(s) ds, \quad t \ge 0. 
 \end{align}
 \end{theorem}

\bigskip

\section{PDE Models} \label{sec-PDE}

One can regard $\bar{\mfI}(t,x)$ as the `distribution function' of $\bar{I}(t) = \bar{\mfI}(t,t+\bar{x})$ over the `ages' $x \in [0,t+\bar{x})$ for each fixed $t$. If $x\mapsto \bar{\mfI}(t,x)$ is absolutely continuous, we denote by $\bar{\mfi}(t,x)
 =\bar{\mfI}_x(t,x)$ the density function of $\bar{\mfI}(t,x)$ with respect to $x$.  
Note that $\bar{S}(t)=0$ for $t<0$ and $\bar{\mfi}(t,x)=0$ both for $t<0$ and $x < 0$.

\subsection{The case $F$ absolutely continuous}
In this subsection, we assume that $F$ is absolutely continuous, $F(dx)=f(x)dx$, and we denote by $\mu(x)$ the hazard function of the r.v. $\eta$, i.e., $\mu(x):=f(x)/F^c(x)$ for $x\ge 0$.
 If the density function $\bar{\mfi}(t,x)$ exists, we obtain the following PDE representation. 
  
\begin{prop} \label{prop-PDE}
Suppose that $F$ is absolutely continuous, with the density $f$, and that $\bar{\mfI}(0,x)$ is differentiable with respect to $x$, with the density function $\bar{\mfi}(0,x)$. Then for $t>0$, the increasing function $\bar{\mfI}(t,\cdot)$ is absolutely continuous, and $(t,x)$ a.e. in $(0,+\infty)^2$, 
\begin{align} \label{eqn-barI-density-PDE-mu}
 \frac{\partial \bar{\mfi}(t,x)}{\partial t} +  \frac{\partial \bar{\mfi}(t,x)}{\partial x}
 & = -\mu(x) \bar{\mfi}(t,x) \,,   
\end{align}
with the  initial condition $\bar{\mfi}(0,x)= \bar{\mfI}_x(0,x)$ for $x \in [0,\bar{x}]$, and
the boundary condition
\begin{equation}\label{BC}
\bar{\mfi}(t,0)=\bar{S}(t)\int_0^{t+\bar{x}}\frac{\bar{\lambda}(x)}{\frac{F^c(x)}{F^c(x-t)}} \bar{\mfi}(t,x)dx\,,
\end{equation}
with the convention that $F^c=1$ on $\R_-$,  and that the integrand in \eqref{BC} is zero when $F^c(x)=0$. 

In addition, 
\begin{equation} \label{eqn-barS-der}
\bar{S}'(t) = -  \bar{\mfi}(t,0),\quad\text{and }\ \bar{S}(0)=1-\bar{I}(0)\,.
\end{equation}

Moreover, the PDE \eqref{eqn-barI-density-PDE-mu} has a unique solution 
which is given as follows. For $x\ge t$, 
\begin{equation}\label{ident1}
 \bar{\mfi}(t,x)=\frac{F^c(x)}{F^c(x-t)}\bar{\mfi}(0,x-t) \, ,
 \end{equation}
 while for $t>x$,
 \begin{equation}\label{ident3}
\bar{\mfi}(t,x)=F^c(x)\bar{\mfi}(t-x,0)\,,
\end{equation}
and the boundary function is the unique solution of the integral equation
\begin{equation} \label{ident2}
\bar{\mfi}(t,0)= \left(  \bar{S}(0)-\int_0^t\bar{\mfi}(s,0)ds \right) \left( \int_0^{\bar{x}} \bar\lambda(y+t)  \bar{\mfi}(0,y)dy   + \int_0^t  \bar\lambda(t-s) \bar{\mfi}(s,0)ds \right)\,.
\end{equation}
\end{prop}

 \begin{remark} 
The PDE \eqref{eqn-barI-density-PDE-mu} can be considered as a linear equation, with a nonlinear boundary condition which is the integral equation \eqref{ident2}. 

It follows from \eqref{ident1} and \eqref{ident3} that $F^c(x)=0$ implies that $\bar{\mfi}(t,x)=0$. This is why we can impose that the integrand in the right hand side of \eqref{BC} is zero whenever $F^c(x)=0$.

We remark that the PDE given in \cite{KM32} resembles that given in \eqref{eqn-barI-density-PDE-mu}, see equations (28)--(29), see also equation (2.2) in \cite{inaba2001kermack}. In particular, the  function $\mu(x)$ is interpreted as the recovery rate at infection age $x$. Equivalently, it is the hazard function of the infectious duration. 
 \end{remark}
 
 \begin{remark}\label{infinite-dim}
 In a sense, what we do in the present paper can be interpreted as follows: we replace the two--dimensional system of equations with memory \eqref{eqn-barS}--\eqref{eqn-overline-cal-I-2} (with, see \eqref{eqn-bar-Upsilon},  $\bar{\Upsilon}(t)$ replaced by $\bar{S}(t)  \overline{\sI}(t)$) by the infinite dimensional system of ODE-PDE  \eqref{eqn-barS}--\eqref{eqn-barI-density-PDE-mu}-\eqref{BC} (with, see again \eqref{eqn-bar-Upsilon}, $\bar{\Upsilon}(t)$ replaced by $\bar{\mfi}(t,0)$).
 
 At the level of our population of size $N$, we have a two--dimensional non--Markov process $(S^N(t),\mathcal{I}^N(t))$.
 For any $t\ge0$, let $\bar{\mfi}^N(t)$ denote the measure whose distribution function is $x\mapsto\bar{\mfI}^N(t,x)$.
 Theorem \ref{thm-FLLN} implies that locally uniformly in $t$, $\bar{\mfi}^N(t)$ converges weakly to the measure which has the density $\bar{\mfi}(t,x)$ w.r.t. Lebesque's measure. $\bar{\mfi}^N(t)$ is a point measure which assigns the mass $N^{-1}$ to any $x$ which is the infection age of one of the individuals  infected at time $t$.
 Clearly, from the knowledge of $\bar{\mfi}^N(t)$, we can deduce the values of both $I^N(0)$ and $A^N(t)$, hence of $S^N(0)$ and of $S^N(t)$ (see \eqref{eqn-Sn-rep}). 
  Note that the points of the measure  $\bar{\mfi}^N(t)$ which are larger than $t$ are the 
  $\{\tilde{\tau}^N_{j,0}+t,  1\le j\le I^N(0)\}$, and those which are less than $t$ are the
  $\{t-\tau^N_i, 1\le i\le A^N(t)\}$. Hence from \eqref{eqn-cI-n}, $\mathcal{I}^N(t)$ is a function of both $\bar{\mfi}^N(t)$
  and the $\lambda_i$'s. The same is true for
   $\Upsilon^N(t)$.  Conditionally upon the $\lambda_i$'s,   the process $\bar{\mfi}^N(t)$
  is a measure-valued Markov process, which evolves as follows. Each point $x$ which belongs to it 
  increases at speed $1$, dies at rate $\mu(x)$, and new points are added at rate $\Upsilon^N(t)$.  $\bar{\mfi}^N(t)$ is determined by a sequence of at most $N$ positive numbers; it can be considered as an element of $\cup_{k=1}^N\R^k$. We have ``Markovianized'' the two--dimensional non--Markov process $(S^N(t),\mathcal{I}^N(t))$, at the price of increasing dramatically the dimension. 
  
  Note that the pair composed of $\bar{\mfi}^N(t)$ and the collection $\{\lambda^0_j, 1\le j\le I^N(0); \lambda_i, 1\le i\le A^N(t)\}$ is a Markov process with values in  $\cup_{k=1}^N(\R\times D)^k$. $\bar{\mfi}^N(t)$ evolves as above, and 
  each new $\lambda_i$ is a random element of $D$ with the same law, independent of everything else.
  
  We expect to write and study the equation for the measure--valued Markov process $\bar{\mfi}^N(t)$ in a future work.
 \end{remark}
 
\begin{remark}\label{rem-lambda-indep1}
Recall the special case in Remark \ref{rem-R0} with  $\lambda_i(t)=\tilde\lambda(t)\bone_{t<\eta_i}$, where $\tilde\lambda(t)$ is a deterministic function.
  Then
 $\bar\lambda(t)=\tilde\lambda(t)F^c(t)$, and $\E\big[\lambda^0(t)|\tilde\tau^N_{0}=y\big]=\tilde\lambda(t+y)\frac{F^c(t+y)}{F^c(y)}$. 
 In that case, the boundary condition in \eqref{BC} becomes 
 $$\bar\mfi(t,0) = \bar{S}(t) \int_0^{t+\bar{x}} \tilde\lambda(x)  \bar{\mfi}(t,x) dx$$
 This is usually how the boundary condition is imposed in the literature of PDE epidemic models (see, e.g., \cite[equation (2.5)]{inaba2001kermack}, \cite[equation (1.1)]{magal2013two} and \cite[equation (2)]{foutel2020individual}).
 This expression has clearly a very intuitive interpretation. 
  $\bar{\mfi}(t,0)$ is the instantaneous rate for an individual to get infected at time $t$ (resulting in a newly infectious individual with a zero age of infection), 
 while the right hand side is the instantaneous infection rate by the existing infectious population at time $t$, which depends on all the infectious individuals with all ages of infection. This of course includes time $t=0$, which formulates a constraint on the initial condition $\{\bar{\mathfrak I}(0,x)\}_{0\le x\le \bar{x}}$. 
\end{remark}

\begin{proof}
By the fact that $F$ has a density, we see that the two partial derivatives of $\bar\mfI$ exist $(t,x)$ a.e. From \eqref{eqn-barI}, they satisfy
\begin{align} \label{barI-partial-t}
 \bar{\mfI}_t(t,x) & = - \bone_{x\ge t} \frac{F^c(x)}{F^c(x-t)}  \bar{\mfI}_x(0,x-t)  -  \int_0^{(x-t)^+} \frac{f(t+y)}{F^c(y)}  \bar{\mfI}_x(0, y)dy   \non\\
& \qquad + \bar{\mfI}_x(t,0) 
-  \bone_{t> x}  F^c(x)\bar{\mfI}_x(t-x,0) - \int_{(t-x)^+}^t f(t-s) \bar{\mfI}_x(s,0) ds,
\end{align}
and
\begin{align} \label{barI-partial-x}
\bar{\mfI}_x(t,x)  &=\bone_{x\ge t}  \frac{F^c(x)}{F^c(x-t)}  \bar{\mfI}_x(0,x-t) 
 +   \bone_{t> x} F^c(x) \bar{\mfI}_x(t-x,0).
\end{align}
Thus, summing up \eqref{barI-partial-t} and \eqref{barI-partial-x}, we obtain for $t>0$ and  $x>0$,
\begin{align} \label{barI-partial-t+x}
 \bar{\mfI}_t(t,x)  +\bar{\mfI}_x(t,x) 
& =-  \int_0^{(x-t)^+} \frac{f(t+y)}{F^c(y)}  \bar{\mfI}_x(0, y)dy   + \bar{\mfI}_x(t,0) 
- \int_{(t-x)^+}^t f(t-s) \bar{\mfI}_x(s,0) ds\,. 
\end{align}

Denote  $ \bar{\mfI}_{x,t}(t,x) = \frac{\partial^2 \bar{\mfI}(t,x)}{\partial x \partial t} = \frac{\partial }{\partial x}  \bar{\mfI}_t(t,x)   $ and 
$ \bar{\mfI}_{x,x}(t,x)= \frac{\partial^2 \bar{\mfI}(t,x)}{\partial x \partial x}$. 
By taking the derivative on both sides of \eqref{barI-partial-t+x} with respect to $x$ (possibly in the distributional sense for each term on the left), we obtain
for $t>0$ and  $x>0$, 
\begin{align} \label{barI-partial-t+x-2}
 \bar{\mfI}_{x,t}(t,x)  +\bar{\mfI}_{x,x}(t,x) 
 &= -  \bone_{x\ge t}  \frac{f(x)}{F^c(x-t)}  \bar{\mfI}_x(0, x-t) -\bone_{t>x} f(x)\bar{\mfI}_{x}(t-x,0). 
\end{align}
Since $ \frac{\partial^2 \bar{\mfI}(t,x)}{\partial x \partial t} =  \frac{\partial^2 \bar{\mfI}(t,x)}{\partial t \partial x}$, we obtain 
 the expression
 \begin{align} \label{eqn-barI-density-PDE}
 \frac{\partial \bar{\mfi}(t,x)}{\partial t} +  \frac{\partial \bar{\mfi}(t,x)}{\partial x}
 & = - \bone_{x\ge t}  \frac{f(x)}{F^c(x-t)}  \bar{\mfi}(0, x-t)  -   \bone_{t> x}  f(x) \bar\mfi(t-x,0) \,. 
\end{align}

As concerns the boundary condition, we note that, given \eqref{eqn-barS-der}, \eqref{ident1} and \eqref{ident3}, \eqref{BC} and \eqref{ident2} are equivalent. Hence we will establish \eqref{ident2},  \eqref{eqn-barS-der}, \eqref{ident1} and \eqref{ident3}.

 For the boundary condition $\bar\mfi(t,0)$,
 by   \eqref{eqn-overline-cal-I-2} and \eqref{eqn-bar-Upsilon}, we have
\begin{align*}
\bar{\mfi}(t,0)= \bar{S}(t) \left( \int_0^{\bar{x}} \bar\lambda(y+t)  \bar{\mfi}(0,y)dy   + \int_0^t  \bar\lambda(t-s) \bar{\mfi}(s,0)ds \right), 
\end{align*}
where by \eqref{eqn-barS}, 
\[  \bar{S}(t) = \bar{S}(0)-\int_0^t\bar{\mfi}(s,0)ds\,.\]
Thus we obtain the expression in \eqref{ident2}. 
 We next prove that equation \eqref{ident2} has a unique non--negative solution. 
 Observe that $u(t) = \bar\mfi(t,0)$ is also a solution to 
 \begin{equation}\label{eq+}
u(t)=\left(\int_0^{\bar{x}} \bar\lambda(y+t)  \bar{\mfi}(0,y)dy+\int_0^t \bar\lambda(t-s)u(s)ds\right)\left(\bar{S}(0)-\int_0^t u(s)ds\right)^+,
\end{equation}
and any non--negative solution of \eqref{ident2} solves \eqref{eq+}.

First, note that since for any $t\ge0$, $0\le \bar{\lambda}(t)\le \lambda^\ast$, 
\begin{equation} \label{eqn-lambda-mfi-bound0}
0\le  \int_0^{\bar{x}} \bar\lambda(y+t)  \bar{\mfi}(0,y)dy \le\lambda^\ast\bar{I}(0),
\end{equation}
from which we conclude that 
$\int_0^tu(s) ds\le\bar{S}(0)$.
 Indeed, if that were not the case, there would exist a time $T_{\bar{S}(0)}<t$ such that 
  $\int_0^{T_{\bar{S}(0)}} u(s) ds=\bar{S}(0)$,
 hence 
 $\int_0^t u(s) ds\ge\bar{S}(0)$ 
and from \eqref{eq+}, we would have $u(t)=0$ for any $t\ge T_{\bar{S}(0)}$, so that indeed $\int_0^tu(s) ds\le\bar{S}(0)$.

 Under Assumption \ref{AS-lambda}, using \eqref{eqn-lambda-mfi-bound0},  if $u_1(t)$ and $u_2(t)$ are two nonnegative integrable solutions, then
 \begin{align*}
 |u_1(t) - u_2(t)|
 & \le  \bar{S}(0) \int_0^t\bar\lambda(t-s)  |u_1(s) - u_2(s) |ds + \lambda^* (\bar{I}(0)+\bar{S}(0) ) \int_0^t | u_1(s) - u_2(s)|ds
\\
& \le 2 \lambda^* \int_0^t   |u_1(s) - u_2(s)|ds\,, 
 \end{align*}
 which, combined with Gronwall's Lemma, implies that $u_1 \equiv u_2$. Now existence is provided by the fact that the function $\bar\mfi(t,0)$ is a non--negative solution of \eqref{eq+}. 
 
 Note also that clearly, using a combination of an argument similar to that used for uniqueness, and of the classical estimate on Picard iterations for ODEs, one could establish that the sequence defined by $u^{(0)}(t)\equiv0$ and
 for $n\ge0$,
 \begin{align*}
u^{(n+1)}(t)= \left(  \bar{S}(0)-\int_0^t u^{(n)}(s)ds \right) \left( \int_0^{\bar{x}} \bar\lambda(y+t)  \bar{\mfi}(0,y)dy   + \int_0^t  \bar\lambda(t-s) u^{(n)}(s)ds \right), 
\end{align*}
given $ \bar{\mfi}(0,\cdot)$, 
is a Cauchy sequence in $C(\RR_+)$, hence existence.

We next derive the explicit solution expressions in \eqref{ident1} and \eqref{ident3}. It follows from 
\eqref{eqn-barI-density-PDE}Â that
for $x\ge t$, $0\le s\le t$,
\[ \frac{\partial \bar{\mfi}}{\partial  s}(s, x-t+s)=-\frac{f(x-t+s)}{F^c(x-t)}\bar{\mfi}(0,x-t),\]
while for $t>x$, $0\le s\le x$,
\[ \frac{\partial \bar{\mfi}}{\partial  s}(t-x+s,s)=-f(s)\bar{\mfi}(t-x,0)\,.\]
Integrating the first identity from $s=0$ to $s=t$, we deduce that for $x\ge t$,
\begin{align*}
\bar{\mfi}(t,x)&=\bar{\mfi}(0,x-t)-\frac{\bar{\mfi}(0,x-t)}{F^c(x-t)}\int_0^t f(x-t+s)ds\\
&=\bar{\mfi}(0,x-t)\left(1-\frac{F(x)-F(x-t)}{F^c(x-t)}\right)\\
&=\frac{F^c(x)}{F^c(x-t)}\bar{\mfi}(0,x-t)\, ,
\end{align*}
so that for $x\ge t$,
\begin{equation}\label{identity1}
\frac{f(x)}{F^c(x-t)}\bar{\mfi}(0,x-t)=\frac{f(x)}{F^c(x)}\bar{\mfi}((t,x)\,.
\end{equation}
Now for $t>x$, we integrate the second identity from $s=0$ to $s=x$, and get
\begin{align*}
\bar{\mfi}(t,x)&=\bar{\mfi}(t-x,0)-\bar{\mfi}(t-x,0)\int_0^x f(s)ds\\
&=F^c(x)\bar{\mfi}(t-x,0)\,.
\end{align*}
Hence for $t>x$, 
\begin{equation}\label{identity2}
f(x)\bar{\mfi}(t-x,0)=\frac{f(x)}{F^c(x)}\bar{\mfi}(t,x)\,.
\end{equation}
Clearly,
 \eqref{ident1} is equivalent to  \eqref{identity1}, \eqref{ident3} is equivalent to \eqref{identity2}, and \eqref{eqn-barI-density-PDE-mu} follows from \eqref{eqn-barI-density-PDE}, \eqref{identity1} and \eqref{identity2}. 
 \end{proof}

 \begin{coro}\label{coroFormula}
 The formula \eqref{eqn-overline-cal-I-2} for $ \overline{\sI}(t) $ can be rewritten 
 \begin{equation}\label{eqFormula}
  \overline{\sI}(t) =\int_0^{t+\bar{x}}\frac{\bar\lambda(y)}{\frac{F^c(y)}{F^c(y-t)}}\bar\mfI(t,dy),
  \end{equation}
  where $F^c(z)=1$, for $z\le0$.
 \end{coro}
 \begin{proof} 
 We first deduce from \eqref{eqn-barI} that for $t>x$, $x\mapsto\bar\mfI(t,x)$ is differentiable, and
$\bar\mfI(t,dx)=F^c(x)\bar\Upsilon(t-x)$, and for fixed $t$, on $[0, \bar{x}]$, the function $y\to \bar{\mathfrak{I}}(t,t+y)$ is of finite total variation and satisfies $\bar{\mathfrak{I}}(t,t+dy) = \frac{F^c(t+y)}{F^c(y)} \bar{\mathfrak{I}}(0,dy)$.
Inserting the resulting formulas for $ \bar{\mathfrak{I}}(0,dy)$ and $\bar\Upsilon$ in the  first and second integrals of the right hand side of \eqref{eqn-overline-cal-I-2}, we obtain
   \begin{align*}
  \overline{\sI}(t)&= \int_0^{\bar{x}} \bar\lambda(t+y)  \bar{\mfI}(0,dy) + \int_0^t \frac{\bar\lambda(x)}{F^c(x)}\bar{\mfI}(t,dx) \\
  &=\int_0^{\bar{x}} \bar\lambda(t+y)\frac{F^c(y)}{F^c(t+y)}\bar\mfI(t,t+dy) + \int_0^t \frac{\bar\lambda(x)}{F^c(x)}\bar{\mfI}(t,dx)\,,
  \end{align*}
  from which the result follows.
 \end{proof}

 \begin{remark} \label{rem-lambda-indep}
In the special case  $\lambda_i(t)=\tilde\lambda(t)\bone_{t<\eta_i}$ as discussed in Remark \ref{rem-lambda-indep1}, \eqref{eqFormula} reduces to the very simple formula
 \begin{equation}  \label{eqn-overline-sI-ind}
  \overline{\sI}(t) =\int_0^{t+\bar{x}}\tilde\lambda(y) \bar\mfI(t,dy)\,.
  \end{equation}
   A similar formula holds if we replace the deterministic function $\tilde\lambda(t)$ by a copy $\lambda_i(t)$ of a random function, which is independent of $\eta_i$, as discussed in Remark \ref{rem-R0}, and whose expectation is $\tilde\lambda(t)$.
 Then, we have
\begin{align} \label{eqn-barUpsilon-ind}
\bar\Upsilon(t) = \bar{S}(t) \overline\sI(t)  &= \bar{S}(t) \int_0^{t+\bar{x}} \tilde\lambda(x)  \bar{\mfI}(t,d x)  \non \\
&= \bar{S}(t) \int_0^{t+\bar{x}} \tilde\lambda(x)  \bar{\mfi}(t,x) dx 
\end{align}
Since $\bar\mfi(t,0) = \bar\Upsilon(t)$, the results above can be stated using this expression of $\bar\Upsilon$.  
 \end{remark}

 In the special case of exponentially distributed infectious periods, i.e. $\mu(x)\equiv\mu$, we obtain the following well known results, see, e.g., \cite{thieme1993may,inaba2004mathematical,magal2013two}.  
\begin{coro} \label{coro-PDE-exp-SIR}
If the c.d.f. $F(t) = 1-e^{-\mu t}$, we have for $t>0$ and $x>0$, 
\begin{align} \label{eqn-barI-density-PDE-exponential}
 \frac{\partial \bar{\mfi}(t,x)}{\partial t} +  \frac{\partial \bar{\mfi}(t,x)}{\partial x}
 = -\mu  \bar{\mfi}(t,x) 
\end{align}
with the initial condition $\bar{\mfi}(0,x)$ given for $x \in [0, \bar{x}]$ and  the boundary condition  for
$\bar{\mfi}(t,0)$ as given in \eqref{ident2}. 
\end{coro}

\begin{proof}
In this case, the above proof simplifies. Indeed, we have for $t\ge 0$ and $x\ge 0$, 
\begin{align*}
\bar{\mfI}(t,x) 
& =  \bone_{x\ge t}e^{-\mu t}\bar{\mfI}(0, x-t)     +  \int_{(t-x)^+}^t e^{-\mu(t-s)} \bar{\Upsilon}(s) ds. 
\end{align*}
By taking derivative with respect to $x$ when $t>0$ and $x>0$, we obtain that  equation \eqref{barI-partial-x} becomes
\begin{align*}
\bar{\mfi}(t,x) &=  \bone_{x\ge t} \bar{\mfi}(0, x-t) e^{-\mu t} +   \bone_{x< t} e^{-\mu x} \bar{\mfI}_x(t-x,0).
\end{align*}
Taking derivatives of this equation with respect to $t$ and $x$, we obtain  for $t>0$ and $x>0$, 
\begin{align*} 
 \frac{\partial \bar{\mfi}(t,x)}{\partial t} +  \frac{\partial \bar{\mfi}(t,x)}{\partial x}
 &=  -  \bone_{x\ge t}\bar{\mfi}(0, x-t) \mu e^{-\mu t} -  \mu e^{-\mu x}   
\bar{\mfI}_x(t-x,0) \\
 &= - \mu  \bar{\mfi}(t,x). 
\end{align*}
The boundary conditions follow in the same way as in the general model. 
\end{proof}

\begin{remark} \label{rem-iid-initial}
If the remaining infectious periods of the initially infectious individuals $\{\eta^0_j, j=1,\dots, I^N(0)\}$ are i.i.d. with c.d.f. $F_0$ instead of  depending on the infection age in \eqref{enq-eta0-age}, then we obtain the limits 
\begin{align*}
\bar{\mfI}(t,x) &=   \bar{\mfI}(0, (x-t)^+) F_0^c(t) +  \int_{(t-x)^+}^t F^c(t-s) \bar{\Upsilon}(s) ds, \\
\bar{R}(t) &= \bar{I}(0) F_0(t) + \int_0^t F(t-s) \bar{\Upsilon}(s) ds,  
\end{align*}
(noting that they are not continuous unless $F_0$ is continuous), 
and assuming the density functions exist, we obtain the PDE: 
\begin{align*}
 \frac{\partial \bar{\mfi}(t,x)}{\partial t} +  \frac{\partial \bar{\mfi}(t,x)}{\partial x}
 =  - \left({\bf1}_{x>t}\mu_0(t)  +{\bf1}_{t>x} \mu(x)\right) \bar\mfi(t,x),
\end{align*}
where $\mu_0(t)=f_0(t)/F^c_0(t)$.
As in the proof of  Proposition \ref{prop-PDE}, we obtain that
 the PDE \eqref{eqn-barI-density-PDE-mu} has a unique solution 
which is given as follows. For $x\ge t$, 
\begin{equation*}
 \bar{\mfi}(t,x)= F_0^c(t) \bar{\mfi}(0,x-t) \, ,
 \end{equation*}
 while for $t>x$,
 \begin{equation*}
\bar{\mfi}(t,x)=F^c(x)\bar{\mfi}(t-x,0)\,,
\end{equation*}
and the boundary function is the unique solution of the integral equation
\begin{equation*}
\bar{\mfi}(t,0)= \left(  \bar{S}(0)-\int_0^t\bar{\mfi}(s,0)ds \right) \left( \int_0^{\bar{x}} \bar\lambda(y+t)  \bar{\mfi}(0,y)dy   + \int_0^t  \bar\lambda(t-s) \bar{\mfi}(s,0)ds \right)\,.
\end{equation*}
\end{remark}

\subsection{The general case}

We now generalize the result of Proposition \ref{prop-PDE} to the case where the distribution $F$ is not absolutely continuous.
We denote below by $\nu$ the law of $\eta$, i.e. the measure whose distribution function is $F$. For reasons which will be explained in Remark \ref{re:leftcont} below, we shall in this subsection use the left continuous versions of $F$ and $F^c$. In order to simplify notations, we define 
\[ G(t)=F(t^-),\ \ \ G^c(t)=1-G(t)=F^c(t^-)\,.\]
\begin{prop} \label{prop-PDE-measure}
Suppose that $\bar{\mfI}(0,x)$ is differentiable with respect to $x$, with the density function $\bar{\mfi}(0,x)$. Then for $t>0$, the increasing function $\bar{\mfI}(t,\cdot)$ is absolutely continuous, 
 and the following identity holds: 
\begin{align} \label{eqn-barI-density-PDE-mu-m}
 \frac{\partial \bar{\mfi}(t,x)}{\partial t} +  \frac{\partial \bar{\mfi}(t,x)}{\partial x}
 & = -\frac{\bar{\mfi}(t,x)}{G^c(x)}\nu(dx) \,,   
\end{align}
(i.e., the distribution which appears on the left hand side of \eqref{eqn-barI-density-PDE-mu-m} equals the measure
which has the density $-\frac{\bar{\mfi}(t,x)}{G^c(x)}$ with respect to the measure $\nu$) 
with the initial condition $\bar{\mfi}(0,x)= \bar{\mfI}_x(0,x)$ for $x \in [0,\bar{x}]$, and
the boundary condition
\begin{equation}\label{BC-m}
\bar{\mfi}(t,0)=\bar{S}(t)\int_0^{t+\bar{x}}\frac{\bar{\lambda}(x)}{\frac{G^c(x)}{G^c(x-t)}} \bar{\mfi}(t,x)dx\,,
\end{equation}
with the convention that $G^c=1$ on $\R_-$, and that the integrand in \eqref{BC-m} is zero whenever $G^c(x)=0$.

In addition, 
\begin{equation} \label{eqn-barS-der-m}
\bar{S}'(t) = -  \bar{\mfi}(t,0),\quad\text{and }\ \bar{S}(0)=1-\bar{I}(0)\,.
\end{equation}

Moreover, the PDE \eqref{eqn-barI-density-PDE-mu-m} has a unique solution 
which is given as follows. For $x\ge t$, 
\begin{equation}\label{ident1-m}
 \bar{\mfi}(t,x)=\frac{G^c(x)}{G^c(x-t)}\bar{\mfi}(0,x-t) \, ,
 \end{equation}
 while for $t>x$,
 \begin{equation}\label{ident3-m}
\bar{\mfi}(t,x)=G^c(x)\bar{\mfi}(t-x,0)\,,
\end{equation}
and the boundary function is the unique solution of the integral equation
\begin{equation} \label{ident2-m}
\bar{\mfi}(t,0)= \left(  \bar{S}(0)-\int_0^t\bar{\mfi}(s,0)ds \right) \left( \int_0^{\bar{x}} \bar\lambda(y+t)  \bar{\mfi}(0,y)dy   + \int_0^t  \bar\lambda(t-s) \bar{\mfi}(s,0)ds \right)\,.
\end{equation}
\end{prop}
\begin{remark}\label{re:leftcont}
The product $\frac{\bar{\mfi}(t,x)}{G^c(x)}\nu(dx)$ can also be rewritten as
\[ \bar{\mfi}(t,x)\times \frac{\nu(dx)}{G^c(x)}\, ,\]
where the second factor can be thought of as the ``hazard measure'', i.e., the generalization of the hazard function, of the r.v. $\eta$. The reason why we want to have $G^c(x)$ in the denominator, and not $F^c(x)$ is the following.
If the support of $\nu$ is $[0,x_{max}]$, and $\nu(\{x_{max}\})>0$, then $F^c(x_{max})=0$, while $G^c(x_{max})>0$
and we need a positive denominator at the point $x_{max}$, since $\nu(\{x_{max}\})>0$.

For consistency, in the present subsection we always choose the left continuous version $G$ (resp. $G^c$) of $F$ (resp. $F^c$). Of course, in the case where $F$ is absolutely continuous this makes no difference.
\end{remark}

\begin{remark}\label{infinite-dim-m}
Remark \ref{infinite-dim} can be extended to the present case of a general distribution function $F$, replacing 
the infinite dimensional system of ODE-PDE  \eqref{eqn-barS}--\eqref{eqn-barI-density-PDE-mu}-\eqref{BC} by 
 \eqref{eqn-barS}--\eqref{eqn-barI-density-PDE-mu-m}-\eqref{BC-m}. 
\end{remark}

\begin{proof}
We first rewrite equation \eqref{eqn-barI} as
\begin{align*}
\bar{\mfI}(t,x) = \int_0^{(x-t)^+}\frac{\bar{\mfI}_x(0,y)}{G^c(y)}G^c(t+y)dy+\int_{(t-x)^+}^t \bar{\mfI}_x(s,0)
G^c(t-s)ds\,.
\end{align*}
Differentiating $\bar{\mfI}(t,x)$ in $x$ can be done exactly as in the proof of Proposition \ref{prop-PDE}. Concerning the differentiation in $t$, the differentiation with respect to $t$ appearing in the integrands $G^c(t+y)$ and $G^c(t-s)$ is now a bit more delicate: those functions have not been assumed to be differentiable. Their derivatives in the distributional sense is a measure, whose bracket with a measurable bounded function makes sense, so that \begin{align*} 
\frac{d}{dt}\int_a^b \frac{\bar{\mfI}_x(0,y)}{G^c(y)}G^c(t+y)dy&=\int_a^b \frac{\bar{\mfI}_x(0,y)}{G^c(y)}\nu(t+dy)\\
&=\int_{a+t}^{b+t} \frac{\bar{\mfI}_x(0,z-t)}{G^c(z-t)}\nu(dz),\\
\frac{d}{dt}\int_a^b G^c(t-s)\bar{\mfI}_x(s,0)ds&=\int_a^b \bar{\mfI}_x(s,0)\nu(t-ds)\\
&=\int_{t-b}^{t-a} \bar{\mfI}_x(t-r,0)\nu(dr)\,.
\end{align*}
As a consequence, the above modification in the proof of Proposition \ref{prop-PDE} yields
\begin{align*}
\bar{\mfI}_t(t,x)+\bar{\mfI}_x(t,x)=\bar{\mfI}_x(t,0)-\int_t^{x\vee t}\frac{\bar{\mfI}_x(0,z-t)}{G^c(z-t)}\nu(dz)
-\int_0^{x\wedge t}\bar{\mfI}_x(t-r,0)\nu(dr)\,.
\end{align*}
Differentiating with respect to $x$ finally yields
\begin{align*}
\frac{\partial\bar{\mfi}(t,x)}{\partial t}+\frac{\partial\bar{\mfi}(t,x)}{\partial x}=-{\bf1}_{x\ge t}\frac{\bar{\mfi}(0,x-t)}{G^c(x-t)}\nu(dx)-{\bf1}_{x<t}\bar{\mfi}(t-x,0)\nu(dx)\,.
\end{align*}
We thus deduce that for $x\ge t$, $0\le s\le t$,
\[ \frac{\partial \bar{\mfi}}{\partial  s}(s, x-t+s)=-\frac{\bar{\mfi}(0,x-t)}{G^c(x-t)}\nu(x-t+ds),\]
while for $t>x$, $0\le s\le x$,
\[ \frac{\partial \bar{\mfi}}{\partial  s}(t-x+s,s)=-\bar{\mfi}(t-x,0)\nu(ds)\,.\]
Let us integrate the first identity on the interval $[0,t)$. 
We get
\begin{align*}
\bar{\mfi}(t,x)&=\bar{\mfi}(0,x-t)\left(1-\frac{\int_{[0,t)}\nu(x-t+ds)}{G^c(x-t)}\right)\\
&=\bar{\mfi}(0,x-t)\left(1-\frac{G(x)-G(x-t)}{G^c(x-t)}\right)\\
&=\frac{G^c(x)}{G^c(x-t)}\bar{\mfi}(0,x-t)\,.
\end{align*}
We conclude that for $x\ge t$,
\[\frac{\bar{\mfi}(0,x-t)}{G^c(x-t)}\nu(dx)=\frac{\bar{\mfi}(t,x)}{G^c(x)}\nu(dx)\,.\]
We finally consider the case $t>x$, and integrate the second identity on the interval $[0,x)$, yielding:
\begin{align*}
\bar{\mfi}(t,x)&=\bar{\mfi}(t-x,0)\left(1-\int_{[0,x)}\nu(ds)\right)\\
&=G^c(x)\bar{\mfi}(t-x,0),
\end{align*}
so that, for $t>x$, 
\[\bar{\mfi}(t-x,0)\nu(dx)=\frac{\bar{\mfi}(t,x)}{G^c(x)}\nu(dx),\]
and we have established \eqref{ident1-m}, \eqref{ident3-m}, as well as \eqref{eqn-barI-density-PDE-mu-m}.
The rest of the proof is the same as that of Proposition \ref{prop-PDE}.
\end{proof} 

The case of a deterministic duration $\eta$ is a particular case of the last Proposition.
\begin{coro} \label{prop-PDE-det}
Suppose that the infectious periods are deterministic and equal to $t_i$, i.e., $F(t)=\bone_{t\ge t_i}$,
$G(t)=\bone_{t>t_i}$.
Then we have
\begin{align} \label{eqn-barI-density-PDE-det}
\frac{\partial \bar{\mfi}(t,x)}{\partial t} + \frac{\partial \bar{\mfi}(t,x)}{\partial x}  
&=  -\delta_{t_i}(x) \bar{\mfi}(t, x)\,, 
\end{align}
with  $\delta_{t_i}(x) $ being the Dirac measure at $t_i$, 
 with the initial condition 
$\bar{\mfi}(0,x) =\partial_x\bar{\mfI}(0,x)$ for $x \in [0,t_i]$, and the boundary condition 
\begin{equation}\label{BC-ti}
\bar{\mfi}(t,0)=\bar{S}(t)\int_0^{(t+t_i)\wedge t_i} \bar{\lambda}(x) \bar{\mfi}(t,x)dx\,,
\end{equation}
Note also that 
the boundary 
function $\bar{\mfi}(t,0)$ solves the following Volterra equation: if $0<t<t_i$, 
\begin{equation}\label{ident2-det} \bar{\mfi}(t,0)=\left(\bar{S}(0)-\int_0^t\bar{\mfi}(s,0)ds\right)\left(\int_t^{t_i}\bar{\lambda}(y)\bar{\mfi}(0,y-t)dy+\int_0^t\bar{\lambda}(t-s)\bar{\mfi}(s,0)ds\right),\end{equation}
and if $t\ge t_i$,
\begin{equation} \label{ident2-det2}
\bar{\mfi}(t,0)=\left(\bar{S}(0)-\int_0^t\bar{\mfi}(s,0)ds\right)\int_0^{t_i}\bar{\lambda}(y)\bar{\mfi}(t-y,0)dy\,,\end{equation}

The PDE \eqref{eqn-barI-density-PDE-det} has a unique solution $\bar{\mfi}(t,x)$, 
which is given as follows. $\bar{\mfi}(t,x)=0$ if $x\ge t_i$. For $t\le x<t_i$, 
\begin{equation} \label{eqn-barI-PDE-det-sol-1}
 \bar{\mfi}(t,x)=\bar{\mfi}(0,x-t)\,,\end{equation}
while for $x<t\wedge t_i$,
\begin{equation}\label{eqn-barI-PDE-det-sol-2}\bar{\mfi}(t,x)= \bar{\mfi}(t-x,0).\end{equation} 
\end{coro}

\begin{remark} \label{rem-finalsize-SIR}
The total fraction of the population infected during the epidemic is given by
\begin{align*}
\Phi= \int_0^\infty \bar\mfi(t,0)dt 
\end{align*}
where $\bar\mfi(t,0)$ is the solution to \eqref{ident2}. 
We also refer the reader to equation (12) in Kaplan \cite{Kaplan2020}, based on his constructed ``Scratch" model. 
\end{remark}

\section{On the SIS model with infection-age dependent infectivity} \label{sec-SIS}

In the SIS model, the infectious individuals become susceptible once they recover. Since $S^N(t) + I^N(t) = N$ for each $t\ge 0$ with a population size $N$, the epidemic dynamics is determined by the process $I^N(t)$ alone, and we have the same representations of the processes $\mfI^N_0(t,x)$ and $\mfI^N_1(t,x)$ in \eqref{eqn-In0-rep} and \eqref{eqn-In1-rep}, respectively, while in the representations of $A^N$ in \eqref{eqn-An-PRM} and $\Upsilon^N$ in \eqref{eqn-Upsilon}, the process $S^N(t)$ is replaced by $S^N(t) = N - I^N(t)$.  The aggregate infectivity process $\sI^N(t)$ is still given by \eqref{eqn-cI-n}. 
The two processes $(\mfI^N, \sI^N)$ determine the dynamics of the SIS epidemic model. 
Under Assumptions \ref{AS-FLLN-Initial} and \ref{AS-lambda}, 
\begin{align}
(\overline{\sI}^N, \bar{\mfI}^{N}) \to (\overline{\sI},\bar{\mfI}) \ \text{in probability, locally uniformly in $t$ and $x$}, \qasq N \to\infty,
\end{align}
where 
\begin{align} 
 \overline{\sI}(t) 
  &= \int_0^{\bar{x}} \bar\lambda(y+t)  \bar{\mfI}(0,dy)  + \int_0^t  \bar\lambda(t-s)  \big(1- \bar{\mfI}(s,\infty) \big) \bar{\sI}(s)ds\,,\label{eqn-overline-sI-SIS} \\
\bar{\mfI}(t,x) 
&=  \int_0^{(x-t)^+} \frac{F^c(t+y)}{F^c(y)}  \bar{\mfI}(0, dy)   +  \int_{(t-x)^+}^t F^c(t-s) \big(1- \bar{\mfI}(s,\infty) \big) \bar{\sI}(s) ds\,,\label{eqn-barI-SIS} 
\end{align}
for $t,x \ge 0$. 
 If $\mfI(0,x)$ is differentiable and $F$ is absolutely continuous, then 
the density function 
 $\bar{\mfi}(t,x)= \frac{\partial \bar{\mfI}(t,x)}{\partial x}$ exists and satisfies again \eqref{eqn-barI-density-PDE-mu}.
The same calculations as in the case of the SIR model lead to \eqref{ident1}, \eqref{ident3} 
and \eqref{BC}.
However, the formula for $\bar{S}(t)$ is different in the case of the SIS model, that is, \eqref{eqn-barS-der} does not hold. 
Instead, we have 
\begin{align}\label{eqn-barS-SIS}
\bar{S}(t)&=1-\bar{I}(t)
=1-\int_0^{\bar{x}}\frac{F^c(t+y)}{F^c(y)}\bar{\mfi}(0,y)dy-\int_0^tF^c(t-s)\bar{\mfi}(s,0)ds\,.
\end{align}
Thus,  the Volterra equation on the boundary reads 
\begin{equation}\label{VoltboundSIS}
\begin{split}
\bar{\mfi}(t,0)&=\left(\int_0^{\bar{x}}\bar{\lambda}(t+y)\bar{\mfi}(0,y)dy
+\int_0^t\bar{\lambda}(t-s)\bar{\mfi}(s,0)ds\right)\\ &\quad\times
\left(1-\int_0^{\bar{x}}\frac{F^c(t+y)}{F^c(y)}\bar{\mfi}(0,y)dy-\int_0^tF^c(t-s)\bar{\mfi}(s,0)ds\right)\,,
\end{split}
\end{equation}
whose form is similar to the one for the SIR model.

It is also clear that if the c.d.f. $ F(t) = 1- e^{-\beta t}$, we have the same PDE for  $\bar{\mfi}(t,x)$ as given in \eqref{eqn-barI-density-PDE-exponential} 
with $\mu(x) = \beta$ and the boundary condition:  
\begin{equation*}
\begin{split}
\bar{\mfi}(t,0)&=\left(\int_0^{\bar{x}}\bar{\lambda}(t+y)e^{-\beta y}\, \bar{\mfi}(0,y)dy
+\int_0^t\bar{\lambda}(t-s)e^{-\beta (t-s)}\, \bar{\mfi}(s,0)ds\right)\\ &\quad\times
\left(1-\int_0^{\bar{x}} e^{-\beta t}\, \bar{\mfi}(0,y)dy-\int_0^t  e^{-\beta (t-s)}  \, \bar{\mfi}(s,0)ds\right)\,.
\end{split}
\end{equation*}

If the c.d.f. $F$ of the infectious period is not absolutely continuous, but $\mfI(0,x)$ is differentiable, then we have essentially the same result as in Proposition \ref{prop-PDE-measure}, except that  \eqref{eqn-barS-der-m} is replaced by \eqref{eqn-barS-SIS},
and \eqref{ident2-m} by \eqref{VoltboundSIS}.

Recall that the standard SIS model has a nontrivial equilibrium point $\bar{I}^*= 1-\beta/\lambda$ if $\beta<\lambda$, where  $\lambda$ is the infection rate (the bar over $\lambda$ is dropped for convenience),  and $1/\beta$ is the mean of the infectious periods. See Section 4.3 in \cite{PP-2020} for the account of the SIS model with general infectious periods. Here we consider the model in the generality of infection-age dependent infectivity.
Note that we provide the explicit expressions for the equilibria below assuming they exist. We do not prove the existence of the limit of  $ \bar{\mfI}(t,x)$ as $t\to\infty$, which we leave  as future work. 

\begin{prop}
Suppose that $ \lim \bar{\mfI}(t,x) \to \bar{\mfI}^\ast(x)$ exists as $t\to\infty$ and $ \bar{I}^\ast =\bar{\mfI}^\ast(\infty)$.
If $R_0=\int_0^\infty \bar\lambda(y)dy\le1$,  $\bar{I}^\ast=0$ (the disease free equilibrium). In the complementary case, $R_0=\int_0^\infty \bar\lambda(y)dy>1$, if $\bar{\mfI}(0,\bar{x})>0$,
\begin{equation} \label{eqn-barI*-formula}
 \bar{I}^\ast=1-\left(\int_0^\infty \bar\lambda(y)dy\right)^{-1} =1-\frac{1}{R_0}\,.
 \end{equation}
The density function $\bar{\mfi}(t,x)$ has an equilibrium $ \bar{\mfi}^*(x)$ in the age of infection $x$, given by
\begin{align} \label{eqn-barI*-x-formula}
 \bar{\mfi}^*(x) =  \frac{d \bar{\mfI}^*(x)}{d x} =  \bar{I}^*  \beta F^c(x),
\end{align}
where $\beta^{-1}=\int_0^\infty F^c(t)dt \in (0,\infty)$ is the expectation of the duration of the infectious period. 
If $F$ has a density $f$, then the equilibrium density $ \bar{\mfi}^*(x)$ satisfies 
\begin{align*}
 \frac{d \, \bar{\mfi}^*(x)}{d x} = -  \bar{I}^*  \beta f(x), \quad  \bar{\mfi}^*(0)=  \bar{I}^*  \beta. 
\end{align*}
\end{prop}

\begin{proof}
The fact that $ \bar{I}^\ast =0$ if $R_0\le1$ and $>0$ if $R_0>1$ follows from 
branching process arguments, and the fact that the start of the epidemic can be approximated by a branching process, see e.g.  section 1.3 in \cite{britton2018stochastic}.
Assume that the equilibrium $\bar{\mfI}^*(x) := \bar{\mfI}(\infty,x)$ exists. 
We deduce from \eqref{eqn-barI-SIS}, combined with \eqref{eqFormula}, 
that $\bar{\mfI}^*(x)$ must satisfy
\begin{align*}
\bar\mfI^\ast(x)&=(1-\bar\mfI^\ast(\infty))\int_0^xF^c(u)du\int_0^\infty\frac{\bar\lambda(y)}{F^c(y)}\bar\mfI^\ast(dy)
\\
&=(1-\bar{I}^\ast)\beta^{-1}F_e(x)\int_0^\infty\frac{\bar\lambda(y)}{F^c(y)} \bar\mfI^\ast(dy),
\end{align*}
where $F_e(x)=\beta \int_0^x F^c(s) ds$, the equilibrium (stationary excess) distribution.  Letting $x\to\infty$ in this formula, we deduce
\begin{equation}\label{eqn-barI*}\bar{I}^\ast=(1-\bar{I}^\ast)\beta^{-1}\int_0^\infty\frac{\bar\lambda(y)}{F^c(y)}\bar\mfI^\ast(dy)\,.\end{equation}
Combining the last two equations, we obtain
\begin{equation}\label{eqn-barmfI*} 
\bar\mfI^\ast(x)=\bar{I}^\ast F_e(x)\,.
\end{equation}
Plugging this formula in the previous identity, we deduce that
\[ \bar{I}^\ast=(1-\bar{I}^\ast)\bar{I}^\ast\int_0^\infty \bar\lambda(y)dy\,.\]
Then the formula \eqref{eqn-barI*-formula} can be directly deduced from this equation. 
The formula \eqref{eqn-barI*-x-formula} follows by  taking the derivative with respect to $x$ in \eqref{eqn-barmfI*}.
\end{proof}

\begin{remark}
If the distribution $F$ is exponential, that is, $F(x) = 1-e^{-\beta x}$, then 
we obtain
 $$\bar{\mfI}^*(x)  =  \bar{I}^* (1-e^{-\beta x}), \quad   \bar{\mfi}^*(x) =  \bar{I}^*  \beta e^{-\beta x}, \qandq  \frac{d\, \bar{\mfi}^*(x)}{d x} = -  \bar{I}^*  \beta^2e^{-\beta x}= -\beta \bar{\mfi}^*(x),
 $$
 where $\bar{I}^*$ is given in \eqref{eqn-barI*-formula}. 
 \end{remark}

 \begin{remark} \label{rem-indep-SIS}
 Suppose that $\lambda_i(t)=\lambda(t)\bone_{t<\eta_i}$, where $\lambda(t)$ is a deterministic function, as in Remark \ref{rem-lambda-indep}. Then $\bar\lambda(t) = \lambda(t) F^c(t)$. If $\lambda(t) \equiv \lambda$ is a constant and $F$ has mean $\beta^{-1}$, then $\bar{I}^\ast$ in \eqref{eqn-barI*-formula}
 \begin{equation} \label{eqn-barI*-formula-cons}
 \bar{I}^\ast=1-\left( \lambda\int_0^\infty F^c(y)dy\right)^{-1} = 1- \beta/\lambda =1-\frac{1}{R_0} \,,
 \end{equation}
 which reduces to the well known result for the standard SIS model with constant rates, assuming $\beta<\lambda$. 
 \end{remark}

\section{Proof of the FLLN} \label{sec-proof-SIR}

In this section, we prove Theorem \ref{thm-FLLN}. We will need the following theorem. A similar pre-tightness criterion can be found in Theorem 3.5.1 in Chapter 6 of \cite{khoshnevisan2002multiparameter}, which extends that in the Corollary on page 83 of \cite{billingsley1999convergence} to the space $C([0,1]^k, \RR)$. Those proofs can be easily extended to the space $D_D$. For the convenience of the reader, we give a proof of the following result in Section \ref{sec:App} below.

\begin{theorem} \label{thm-DD-conv0}
Let $\{X^N: N \ge 1\}$ be a sequence of random elements in $D_D$.
If the following two conditions are satisfied: for any $T,S>0$, 
\begin{itemize}
\item[(i)] for any $\ep>0$, $  \sup_{t \in [0,T]}\sup_{s\in [0,S]} \P \big(  |X^N(t, s)|> \ep \big) \to 0$ as $N\to\infty$, and
\item[(ii)] for any $\ep>0$, as $\delta\to0$,
\begin{align*} 
& \limsup_{N\to\infty}  \sup_{t\in [0,T]} \frac{1}{\delta}  \P \bigg(  \sup_{u \in [0,\delta]}\sup_{s \in [0,S]} |X^N(t+u,s) - X^N(t,s)| > \ep\bigg) \to 0, \\
& \limsup_{N\to\infty}  \sup_{s\in [0,S]} \frac{1}{\delta}  \P \bigg(  \sup_{v \in [0,\delta]}\sup_{t \in [0,T]} |X^N(t,s+v) - X^N(t,s)| > \ep\bigg) \to 0, 
\end{align*}
\end{itemize}
then  $X^N(t,s)\to 0 $ in probability, locally uniformly in $t$ and $s$, as $N\to\infty$. 
\end{theorem}

We shall also use repeatedly the following Lemma.
\begin{lemma}\label{le:Portmanteau}
Let $f\in D(\R_+)$ and $\{g_N\}_{N\ge1}$ be a sequence of elements of $D_\uparrow(\R_+)$ which is such that
$g_N\to g$ locally uniformly, where $g\in C_\uparrow(\R_+)$. Then for any $T>0$,
\[ \int_{[0,T]} f(t) g_N(dt)\to \int_{[0,T]} f(t)g(dt)\, .\]
\end{lemma}
\begin{proof}
The assumption implies that the sequence of measures $g_N(dt)$ converges weakly, as $N\to\infty$,
towards the measure $g(dt)$. Since moreover $f$ is bounded, and the set of discontinuities of $f$ is
of $g(dt)$ measure $0$,  this is essentially a minor improvement of the Portmanteau theorem, see 
\cite{billingsley1999convergence}.
\end{proof}

\subsection{Convergence of $\mfI^{N}_0(t,x) $}

We first treat the process $\mfI^{N}_0(t,x) $ in \eqref{eqn-In0-rep}. 

\begin{lemma} \label{lem-barIN0-conv}
Under Assumption \ref{AS-FLLN-Initial},
\begin{equation}
\bar{\mfI}^N_0(t,x) \to \bar{\mfI}_0(t,x) \qinq D_D \qasq N \to \infty,
\end{equation}
in probability, where the limit $\bar{\mfI}_0(t,x)$ is given by
\begin{equation} \label{eqn-bar-mfI-0-rep1}
\bar{\mfI}_0(t,x) :=   \int_0^{(x-t)^+} \frac{F^c(t+y)}{F^c(y)}  \bar{\mfI}(0, dy) , \quad t, x \ge 0.
\end{equation}
\end{lemma}

\begin{proof}
Recall that 
\begin{align*}
\bar{\mfI}^{N}_0(t,x) = N^{-1}\sum_{j=1}^{I^N(0)} \bone_{\eta_j^0 > t} \bone_{\tilde{\tau}_{j,0}^N  \le (x-t)^+ } = N^{-1} \sum_{j=1}^{\mfI^N(0, (x-t)^+)} \bone_{\eta_j^0 > t}\,.
\end{align*}
Note that the pair of variables $(\tilde{\tau}_{j,0}^N, \eta^0_j)$ satisfies \eqref{enq-eta0-age}, and $\mfI^N(0, (x-t)^+) = \max\{j\ge1: \tilde{\tau}_{j,0}^N  \le (x-t)^+\}$.
Let
\begin{align} \label{eqn-wt-mfI-0}
\wt{\mfI}^{N}_0(t,x) = N^{-1}\sum_{j=1}^{\mfI^N(0, (x-t)^+)}  \frac{F^c(t+ \tilde{\tau}_{j,0}^N  )}{F^c(\tilde{\tau}_{j,0}^N  )}  = \int_0^{(x-t)^+} \frac{F^c(t+y)}{F^c(y)}  \bar\mfI^N(0, dy)  \,.
\end{align}

We will first show that $\wt{\mfI}^N_0(t,x) \to \bar{\mfI}_0(t,x)$ (this will be step 1 of the proof), and then that 
 $\bar{\mfI}^{N}_0(t,x) - \wt{\mfI}^N_0(t,x)\to0$ (this will be step 2 of the proof), both in probability, locally uniformly in $t$ and $x$, as $N\to\infty$.
 
{\bf Step 1} We show that, as $N\to\infty$, 
\begin{equation}\label{conv0tilde}
\wt{\mfI}^N_0(t,x) \to \bar{\mfI}_0(t,x) \ \text{ in probability, locally uniformly in $t$ and $x$}. 
\end{equation}
From Lemma \ref{le:Portmanteau}, Assumption \ref{AS-FLLN-Initial} and the continuous mapping theorem, we deduce that for any $t,x \ge0$,
$\wt{\mfI}^N_0(t,x) \to \bar{\mfI}_0(t,x)$ in probability, as $N\to\infty$. It thus remains to show that the sequence $\{X^N:=\wt{\mfI}^N_0-\bar{\mfI}_0,\ N\ge1\}$ satisfies condition (ii) in Theorem \ref{thm-DD-conv0}. 
In fact, it is easily seen that it is sufficient to verify condition (ii) with $X^N=\wt{\mfI}^N_0$. Indeed, both
\[ \sup_{0\le u\le \delta}\sup_{0\le x\le \bar{x}}|\bar{\mfI}_0(t+u,x)-\bar{\mfI}_0(t,x)|\ \text{ and }
\sup_{0\le v\le \delta}\sup_{0\le t\le T}|\bar{\mfI}_0(t,x+v)-\bar{\mfI}_0(t,x)|\]
tend to $0$, as $\delta\to0$, which is an easy consequence of the computations which  follow. Let us now consider $X^N=\wt{\mfI}^N_0$.
We have 
\begin{align*}
& \wt{\mfI}^N_0(t+u,x)-\wt{\mfI}^N_0(t,x) \\
&=\int_0^{(x-t-u)^+}\frac{F^c(t+u+y)}{F^c(y)}\bar\mfI^N(0, dy)
-\int_0^{(x-t)^+}\frac{F^c(t+y)}{F^c(y)}\bar\mfI^N(0, dy)\\
&=\int_0^{(x-t-u)^+}\frac{F^c(t+u+y)-F^c(t+y)}{F^c(y)}\bar\mfI^N(0, dy)-\int_{(x-t-u)^+}^{(x-t)^+}\frac{F^c(t+y)}{F^c(y)}\bar\mfI^N(0, dy)
\end{align*}
which gives 
\begin{align*}
& \big|\wt{\mfI}^N_0(t+u,x)-\wt{\mfI}^N_0(t,x)\big| \\
&\le \int_0^{(x-t-u)^+}\frac{F^c(t+y)-F^c(t+u+y)}{F^c(y)}\bar\mfI^N(0, dy)+\int_{(x-t-u)^+}^{(x-t)^+}\frac{F^c(t+y)}{F^c(y)}\bar\mfI^N(0, dy)\,.
\end{align*}
Consequently,
\begin{align*}
\sup_{0\le u\le\delta,0\le x\le\bar{x}}|\wt{\mfI}^N_0(t+u,x)-\wt{\mfI}^N_0(t,x)|&\le\int_0^{(\bar{x}-t)^+}
\frac{F^c(t+y)-F^c(t+\delta+y)}{F^c(y)}\bar\mfI^N(0, dy)\\
&\quad+ \sup_{0\le x\le\bar{x}}\int_{(x-t-\delta)^+}^{(x-t)^+} \frac{F^c(t+y)}{F^c(y)}\bar\mfI^N(0, dy)\,.
\end{align*}
The limit in probability of the first term on the right of the last inequality equals 
\[ \int_0^{(\bar{x}-t)^+}\frac{F^c(t+y)-F^c(t+\delta+y)}{F^c(y)}\bar\mfI(0, dy),\]
which tends to $0$ as $\delta\to0$, since $F^c$ is continuous on the right and the integrand is between $0$ and $1$.
The second term on the right of the above inequality is nonnegative and upper bounded by
\[ \sup_{0\le x\le\bar{x}}\left(\bar\mfI^N(0, (x-t)^+)-\bar\mfI^N(0, (x-t-\delta)^+)\right),\]
which converges in probability towards
\[ \sup_{0\le x\le\bar{x}}\left(\bar\mfI(0, (x-t)^+)-\bar\mfI(0, (x-t-\delta)^+)\right),\]
and this last expression tends to $0$ as $\delta\to0$. Combining the above arguments, we deduce that
for $\ep>0$, if $\delta>0$ is small enough,
\[\limsup_N\P\left(\sup_{0\le u\le\delta, 0\le x\le \bar{x}}\big|\wt{\mfI}^N_0(t+u,x)-\wt{\mfI}^N_0(t,x)\big|>\ep\right)=0\,.\]
We next consider
\begin{align*}
\wt{\mfI}^N_0(t,x+v)-\wt{\mfI}^N_0(t,x)&=\int_{(x-t)^+}^{(x+v-t)^+}\frac{F^c(t+y)}{F^c(y)}\bar\mfI^N(0, dy)\\
&\le \bar\mfI^N(0,(x+v-t)^+)-\mfI^N(0,(x-t)^+)\,. 
\end{align*}
Hence, 
\begin{align*}
\sup_{0\le v\le\delta,0\le t\le T}\big|\wt{\mfI}^N_0(t,x+v)-\wt{\mfI}^N_0(t,x)\big|=\sup_{0\le t\le T}\left(\bar\mfI^N(0,(x+\delta-t)^+)-\mfI^N(0,(x-t)^+)\right)\,.
\end{align*}
The term on the right of the last inequality converges in probability as $N\to\infty$, towards
\[ \sup_{0\le t\le T}\left(\bar\mfI(0,(x+\delta-t)^+)-\mfI(0,(x-t)^+)\right),\]
which tends to $0$ as $\delta$ tends to $0$. 
Again we easily deduce from these computations that for any $\ep>0$, if $\delta>0$ is small enough,
\[\limsup_N\P\left(\sup_{0\le v\le\delta,0\le t\le T}\big|\wt{\mfI}^N_0(t,x+v)-\wt{\mfI}^N_0(t,x)\big|>\ep\right)=0\,.\]
 We have established \eqref{conv0tilde}.

{\bf Step 2} We finally show that $V^N(t,x):= \bar{\mfI}^{N}_0(t,x) - \wt{\mfI}^{N}_0(t,x)$ satisfies the two conditions of  Theorem \ref{thm-DD-conv0}. 
We have
\begin{align*}
V^N(t,x) 
= N^{-1} \sum_{j=1}^{\mfI^N(0, (x-t)^+)} \left(\bone_{\eta_j^0 > t} - \frac{F^c(t+ \tilde{\tau}_{j,0}^N  )}{F^c(\tilde{\tau}_{j,0}^N  )}  \right).
\end{align*}

We first check condition (i) from Theorem \ref{thm-DD-conv0}. 
We have 
\begin{align*}
\E\big[V^N(t,x)^2\big] 
&= \E \Bigg[ N^{-2} \sum_{j=1}^{\mfI^N(0, (x-t)^+)} \left(\bone_{\eta_j^0 > t} - \frac{F^c(t+ \tilde{\tau}_{j,0}^N  )}{F^c(\tilde{\tau}_{j,0}^N  )}  \right)^2 \Bigg]  \\
& \quad + \E\Bigg[ N^{-2} \sum_{j, j'=1, \, j \neq j'}^{\mfI^N(0, (x-t)^+)} \left(\bone_{\eta_j^0 > t} - \frac{F^c(t+ \tilde{\tau}_{j,0}^N  )}{F^c(\tilde{\tau}_{j,0}^N  )}  \right)  \left(\bone_{\eta_{j'}^0 > t} - \frac{F^c(t+ \tilde{\tau}_{j',0}^N  )}{F^c(\tilde{\tau}_{j',0}^N  )}  \right)\Bigg]  \\
& = N^{-1} \E \Bigg[\int_0^{(x-t)^+}  \frac{F^c(t+ s)}{F^c(s)} \bigg( 1-  \frac{F^c(t+ s)}{F^c(s)}\bigg) \bar\mfI^N(0,ds) \Bigg]
\end{align*}
where the second term in the first equality is equal to zero by the independence of $\eta^0_j$ and $\eta^0_{j'}$ given the times $\tilde{\tau}_{j,0}^N$ and $\tilde{\tau}_{j',0}^N$ and by using a conditioning argument. 
This implies that as $N\to \infty$,  \[\sup_{t\ge0}\sup_{x\ge 0} \E\big[V^N(t,x)^2\big] \to 0,\] and thus condition (i) in Theorem \ref{thm-DD-conv0} holds. 

We next show condition (ii) from Theorem \ref{thm-DD-conv0}, that is, 
 for any $\ep>0$, as $\delta \to 0$, 
\begin{align} \label{eqn-Vn-diff-conv-1}
\limsup_N \sup_{t \in [0,T]} \frac{1}{\delta}\P \left( \sup_{u \in [0,\delta]} \sup_{x \in [0,T']} \big|V^N(t+u, x) - V^N(t,x) \big|>\ep\right) \to 0,
\end{align}
and
\begin{align}\label{eqn-Vn-diff-conv-2}
\limsup_N \sup_{x \in [0,T']} \frac{1}{\delta}\P \left( \sup_{v \in [0,\delta]} \sup_{t \in [0,T]} \big|V^N(t, x+v) - V^N(t,x) \big|>\ep\right) \to 0. 
\end{align}
We first prove \eqref{eqn-Vn-diff-conv-1}. 
We have 
\begin{align} \label{eqn-Vn-diff-conv-1-p1}
& \big|V^N(t+u, x) - V^N(t,x) \big|  \non\\
 &= \Bigg|N^{-1} \sum_{j=1}^{ {\mfI}^N(0, (x-t-u)^+)}\bigg(\bone_{\eta_j^0 > t+u} - \frac{F^c(t+u+ \tilde{\tau}_{j,0}^N  )}{F^c(\tilde{\tau}_{j,0}^N  )}  \bigg)    - N^{-1} \sum_{j=1}^{ {\mfI}^N(0, (x-t)^+)}\bigg(\bone_{\eta_j^0 > t} - \frac{F^c(t+ \tilde{\tau}_{j,0}^N  )}{F^c(\tilde{\tau}_{j,0}^N  )}  \bigg)   \Bigg| \non \\
& \le  \Bigg|N^{-1} \sum_{j=1}^{ {\mfI}^N(0, (x-t-u)^+)}\bigg(\bone_{t<\eta_j^0 \le t+u} - \frac{F^c(t+ \tilde{\tau}_{j,0}^N  ) - F^c(t+ u+\tilde{\tau}_{j,0}^N  )}{F^c(\tilde{\tau}_{j,0}^N  )}  \bigg)  \Bigg|  \non \\
 & \quad + \Bigg|N^{-1} \sum_{j= {\mfI}^N(0, (x-t-u)^+)+1}^{ {\mfI}^N(0, (x-t)^+)}\bigg(\bone_{\eta_j^0 > t} - \frac{F^c(t+ \tilde{\tau}_{j,0}^N  )}{F^c(\tilde{\tau}_{j,0}^N  )}  \bigg)  \Bigg|  \non \\
& \le  N^{-1} \sum_{j=1}^{ {\mfI}^N(0, (x-t-u)^+)}\bone_{t<\eta_j^0 \le t+u}  +  \int_0^{(x-t-u)^+}\frac{F^c(t+s  ) - F^c(t+u+s  )}{F^c(s )}  \bar{\mfI}^N(0,d s) \non \\
& \qquad +  \big|  \bar{\mfI}^N(0, (x-t)^+) - \bar{\mfI}^N(0, (x-t-u)^+) \big|\,. 
\end{align}
For the first term, 
\begin{align} \label{eqn-Vn-diff-conv-1-p2}
& \P \Bigg( \sup_{u \in [0,\delta]} \sup_{x \in [0,T']} N^{-1} \sum_{j=1}^{ {\mfI}^N(0, (x-t-u)^+)}\bone_{t<\eta_j^0 \le t+u}  >\ep/3\Bigg) \non \\
& \le \P \Bigg(  N^{-1} \sum_{j=1}^{ {\mfI}^N(0, (T'-t)^+)}\bone_{t<\eta_j^0 \le t+\delta}  >\ep/3\Bigg) \non \\
& \le \P \Bigg(  N^{-1} \sum_{j=1}^{ {\mfI}^N(0, (T'-t)^+)}\bigg[\bone_{t<\eta_j^0 \le t+\delta}
- \frac{F^c(t+ \tilde{\tau}_{j,0}^N  ) - F^c(t+ \delta+\tilde{\tau}_{j,0}^N  )}{F^c(\tilde{\tau}_{j,0}^N  )}\bigg] >\ep/6\Bigg) \non \\
&\quad+\P \Bigg(  \int_0^{(T'-t)^+}    \frac{F^c(t+s  ) - F^c(t+ \delta+s )}{F^c(s)}   \bar{\mfI}^N(0, ds)>\ep/6\Bigg)
\end{align}

By the conditional independence of the $\eta^0_j$'s, the first term on the right of 
\eqref{eqn-Vn-diff-conv-1-p2}
 is bounded by
\begin{align*}
& \frac{36}{\ep^2 }N^{-1}  \E\left[
\int_0^{(T'-t)^+}    \frac{F^c(t+s  ) - F^c(t+ \delta+s )}{F^c(s)} \left( 1-  \frac{F^c(t+s  ) - F^c(t+ \delta+s )}{F^c(s)} \right)  \bar{\mfI}^N(0, ds)    \right]  \\
& \le 
\frac{36}{\ep^2 }N^{-1}  \E\left[
\int_0^{(T'-t)^+}    \frac{F^c(t+s  ) - F^c(t+ \delta+s )}{F^c(s)}   \bar{\mfI}^N(0, ds)    \right] ,
\end{align*}
which converges to zero as $N\to \infty$. Since by Assumption \ref{AS-FLLN-Initial} $ \bar{\mfI}(0,\cdot)$ is continuous, thanks to Lemma \ref{le:Portmanteau}, $\limsup_N$ of the second term is upper bounded by
\[{\bf1}\bigg\{\int_0^{(T'-t)^+}    \frac{F^c(t+s  ) - F^c(t+ \delta+s )}{F^c(s)}  \bar{\mfI}(0, ds) \ge\ep/6\bigg\},\]
which is zero for $\delta>0$ small enough (clearly uniformly over $t \in [0,T]$).

The second term on the right of \eqref{eqn-Vn-diff-conv-1-p1} is treated exactly as the last term we have just analyzed. Finally for the third term, we note that 
\begin{align*}
& \P \left( \sup_{u \in [0,\delta]} \sup_{x \in [0,T']}  \big| \bar{\mfI}^N(0, (x-t)^+)- \bar{\mfI}^N(0, (x-t-u)^+)\big |  >\ep/3 \right)  \\
& =  \P \left( \sup_{x \in [0,T']} \big | \bar{\mfI}^N(0, (x-t)^+) - \bar{\mfI}^N(0, (x-t-\delta)^+)\big|  >\ep/3 \right).
\end{align*}
Thanks to Assumption \ref{AS-FLLN-Initial}, the $\limsup_N$ of this probability is upper bounded by 
\[
{\bf 1} \left\{ \sup_{x \in [0,T']}  \big|\bar{\mfI}(0, (x-t)^+)- \bar{\mfI}(0, (x-t-\delta)^+) \big|  \ge\ep/3 \right\}
\]
which is zero for $\delta>0$ small enough, since $ \bar{\mfI}(0,\cdot)$ is continuous. The uniformity over $t\in [0,T]$ is obvious. Thus we have shown \eqref{eqn-Vn-diff-conv-1}.

We next prove  \eqref{eqn-Vn-diff-conv-2}. 
Observe that
\[
 V^N(t, x+v) - V^N(t,x) = N^{-1} \sum_{j={\mfI}^N(0, (x-t)^+)+1}^{ {\mfI}^N(0, (x+v-t)^+)} \left(\bone_{\eta_j^0 > t} - \frac{F^c(t+ \tilde{\tau}_{j,0}^N  )}{F^c(\tilde{\tau}_{j,0}^N  )}  \right)
\]
from which we obtain 
\begin{align*}
& \P \left( \sup_{v \in [0,\delta]} \sup_{t \in [0,T]} \big|V^N(t, x+v) - V^N(t,x) \big|>\ep\right) \\
& \le \P \left( \sup_{v \in [0,\delta]} \sup_{t \in [0,T]} \big| \bar{\mfI}^N(0, (x+v-t)^+) - \bar{\mfI}^N(0, (x-t)^+)\big|>\ep\right) \\
& \le  \P \left( \sup_{t \in [0,T]} \big| \bar{\mfI}^N(0, (x+\delta-t)^+) - \bar{\mfI}^N(0, (x-t)^+)\big|>\ep\right)\,.
\end{align*}
Then following the same argument as for the second term on the right of \eqref{eqn-Vn-diff-conv-1-p1}, we can conclude  \eqref{eqn-Vn-diff-conv-2}. 
\end{proof}

\subsection{Convergence of $\bar{\mfI}^{N}_1$}

We first write the process $A^N$ as 
\begin{align}\label{eqn-An-rep}
A^N(t) = M_A^N(t)  +  \Lambda^N(t), 
\end{align}
where 
\begin{align} \label{eqn-Lambda}
\Lambda^N(t) : =  \int_0^t \Upsilon^N(s) ds,
\end{align}
and
\begin{align} \label{eqn-Mn-rep}
M_A^N(t)  =  \int_0^t \int_0^\infty \bone_{u \le \Upsilon^N(s^-) } \overline{Q}(ds,du) , 
\end{align}
where $\overline{Q}(ds,du) = Q(ds,du) - ds du$ is the compensated PRM.

\begin{lemma} \label{lem-Martingale}
Under Assumption \ref{AS-lambda},  the process $\{M^N_A(t): t\ge 0\}$ is a square-integrable martingale with respect to the filtration $\sF^N_A= \{\sF^N_A(t): t\ge0\}$ where 
\begin{align*}
\sF^N_A(t) &:=  \sigma\big\{ I^N(0), \tilde{\tau}^N_{j}: j =1,\dots, I^N(0)\big\} \vee \sigma \big\{\lambda^0_j(\cdot)_{j\ge1}, \lambda_i(\cdot)_{i\ge 1} \big\} \\
 &\qquad \vee \sigma\bigg\{ \int_0^{t'} \int_0^\infty  \bone_{u \le \Upsilon^N(s^{-})} Q(ds,du):  0 \le t' \le t \bigg\}. 
\end{align*}
The quadratic variation of $M^N_A(t)$ is given by
\begin{equation} \label{eqn-MA-qv}
\langle M^N_A \rangle(t) = \Lambda^N(t), \quad t \ge 0. 
\end{equation}
\end{lemma}

\begin{proof}
It is clear that $M^N_A(t) \in \sF^N_A(t)$, and  $\E[|M_A^N(t)|] \le 2 \E[\Lambda^N(t)] \le 2 \lambda^* Nt <\infty$ for each $t\ge 0$, under Assumption \ref{AS-lambda}.
It suffices to verify the martingale property:  for $t_2>t_1\ge 0$,
$$
\E\big[M^N_A(t_2)- M^N_A(t_1)\big|\sF^N_A(t_1) \big] =0 
$$
which can be checked using  the above definition of the filtration. 
In addition, $\E[(M_A^N(t))^2] = \E[\Lambda^N(t)] \le  \lambda^* N t<\infty$ for each $t\ge 0$. The rest is standard. 
\end{proof}

Recall that $(\bar{A}^N, \bar{S}^N, \bar{\Upsilon}^N) := N^{-1}(A^N, S^N, \Upsilon^N) $.
\begin{lemma} \label{lem-tight}
 Under Assumptions \ref{AS-FLLN-Initial} and  \ref{AS-lambda},  the  sequence of processes $\{(\bar{A}^N, \bar{S}^N): N \in\NN\}$ is tight in $D^2$.  The limit of each convergence subsequence of $\{\bar{A}^N\}$, denoted by $\bar{A}$, satisfies
\begin{equation} \label{eqn-An-conv1}
\bar{A}= \lim_{N\to\infty} \bar{A}^N = \lim_{N\to\infty}  \int_0^{\cdot} \bar{\Upsilon}^N(u) du, 
\end{equation}
and 
\begin{equation} \label{eqn-barPhi-bound}
0 \le \int_s^{t} \bar{\Upsilon}^N(u) du \le \lambda^*(t-s),  \quad \text{w.p.\,1} \qforq 0 \le s \le t.
\end{equation}
\end{lemma}

\begin{proof}
It is clear that under Assumption \ref{AS-lambda}, if $\bar{\Lambda}^N(t):=\int_0^t \bar{\Upsilon}^N(u) du$,
 $\bar{\Lambda}^N(0)=0$ and 
\begin{align}
0\le\bar{\Lambda}^N(t) - \bar{\Lambda}^N(s) \le  \lambda^* (t-s), \quad \text{w.p.\,1} \qforq 0 \le s \le t.
\end{align}
Since
\[ \langle\bar{M}^N_A\rangle_t\le N^{-1}\lambda^\ast t,\]
it follows from Doob's inequality that $\bar{M}^N_A(t)$ tends to $0$ in probability, locally uniformly in $t$.
The tightness of $\{\bar{A}^N: N \in\NN\}$
in $D$ follows. Since $\bar{S}^N = \bar{S}^N(0) - \bar{A}^N$ and $\bar{S}^N(0) \RA \bar{S}(0)$ from Assumption \ref{AS-FLLN-Initial}, we obtain the tightness of $\{\bar{S}^N: N \in\NN\}$
in $D$, and thus the claim of the lemma. 
\end{proof}

In the following of this section, we consider  a convergent subsequence of $\bar{A}^N$. 

Recall that 
\begin{align*}
\bar{\mfI}^{N}_1(t,x) = N^{-1} \sum_{i=A^N((t-x)^+)}^{A^N(t)} \bone_{\tau^N_i + \eta_i >t}, \quad t, x\ge 0.
\end{align*}

\begin{lemma} \label{lem-mfI-1-conv}
Under Assumptions \ref{AS-FLLN-Initial} and \ref{AS-lambda}, along a subsequence of $\bar{A}^N$ 
which converges weakly to $\bar{A}$, 
\begin{equation} \label{eqn-barfrakI-conv}
\bar{\mfI}^N_1(t,x) \Rightarrow \bar{\mfI}_1(t,x) \qinq D_D \qasq N \to \infty,
\end{equation}
where the limit $\bar{\mfI}_1(t,x)$ is given by
\begin{equation} \label{eqn-bar-mfI-1-rep1}
\bar{\mfI}_1(t,x) :=  \int_{(t-x)^+}^t F^c(t-s) d \bar{A}(s), \quad t, x \ge 0.
\end{equation}
\end{lemma}

\begin{proof}
Let 
$$
\breve{\mfI}^N_1(t,x) := 
N^{-1} \sum_{i=A^N((t-x)^+)}^{A^N(t)} F^c(t-\tau^N_i) , \quad t, x\ge 0.
$$
We can write (from now on, $\int_a^b$ stands for $\int_{(a,b]}$)
\begin{equation} \label{eqn-breve-mfI-1-rep}
\breve{\mfI}^N_1(t,x)  = \int_{(t-x)^+}^t F^c(t-s) d \bar{A}^N(s). 
\end{equation}
Then from Lemma \ref{le:Portmanteau}, we deduce that for any $t,x\ge0$,
\begin{equation}\label{pointlimit}
\breve{\mfI}^N_1(t,x) \Rightarrow \bar{\mfI}_1(t,x)  \qasq N \to \infty\,.
\end{equation}
We will next show that for any $\ep>0$, there exists $\delta>0$ such that the following holds for any $(t,x)$:
\begin{align}\label{estimbreve}
\limsup_N\P\left(\sup_{t\le t'\le t+\delta, x\le x'\le x+\delta} \left|\breve{\mfI}^N_1(t,x)-\breve{\mfI}^N_1(t',x')\right|>\ep  \right)=0,
\end{align} 
It is not hard to deduce from \eqref{pointlimit} and \eqref{estimbreve}, by a two--dimensional extension of the argument of the Corollary on page 83 of \cite{billingsley1999convergence}, that as $N\to\infty$,
$\breve{\mfI}^N_1(t,x) \Rightarrow \bar{\mfI}_1(t,x)$ locally uniformly in $t$ and $x$. 
 Whenever $t\le t'\le t+\delta$ and $x\le x'\le x+\delta$, we have
\begin{align*}
\big|\breve{\mfI}_1^N(t,x)-\breve{\mfI}_1^N(t',x')\big|&\le\int_{(t-x)^+}^t\big[F^c(t-s)-F^c(t'-s)\big]d\bar{A}^N(s)
+2\sup_{ 0\le t_2- t_1\le2\delta}\big[\bar{A}^N(t_2)-\bar{A}^N(t_1)\big]\,.
\end{align*}
Since $\bar{A}^N(t)\Rightarrow \int_0^t\bar{\Upsilon}(s)ds$ locally uniformly in $t$, and $\bar{\Upsilon}(s)\le\lambda^\ast$, the limit in law of the right hand side of the last inequality is bounded by
\[ \lambda^\ast \int_{(t-x)^+}^t\sup_{t\le t'\le t+\delta}\big[F^c(t-s)-F^c(t'-s)\big]ds+4\lambda^\ast\delta,\]
which is less than $\ep$ for $\delta>0$ small enough. Hence, \eqref{estimbreve} follows.

Let now $$
Y^N(t,x) := \bar{\mfI}^N_1(t,x) - \breve{\mfI}^N_1(t,x) =  N^{-1} \sum_{i=A^N((t-x)^+)}^{A^N(t)}\big(  \bone_{\tau^N_i + \eta_i >t} -F^c(t-\tau^N_i) \big), \quad t,x \ge 0. 
$$

To prove \eqref{eqn-barfrakI-conv}, it remains to show that, as $N\to\infty$, 
\begin{equation} \label{eqn-Yn-conv-0}
Y^N \to 0 \qinq D_D \quad\text{in probability}.
\end{equation}
We apply Theorem \ref{thm-DD-conv0}. By Markov's inequality and the decomposition of $A^N(t)$ in \eqref{eqn-An-rep} with $\E[M_A^N(t)]=0$,   we obtain
\begin{align*}
\P \left(Y^N(t,x)  >\ep  \right) & \le \frac{1}{\ep^2} \E \left[Y^N(t,x)^2 \right] \\
& = \frac{1}{\ep^2 N} \E \left[ \int_{(t-x)^+}^t F(t-s) F^c(t-s) d \bar{A}^N(s) \right] \\
& = \frac{1}{\ep^2 N} \E \left[ \int_{(t-x)^+}^t F(t-s) F^c(t-s) \bar{\Upsilon}^N(s) ds \right] \\
& \le     \frac{1}{\ep^2 N}   \lambda^* \int_{(t-x)^+}^t F(t-s) F^c(t-s) ds \\
\sup_{t\in[0,T],x\in[0,T']}\P \left(Y^N(t,x)  >\ep  \right) & \to 0 \qasq N \to \infty. 
\end{align*}
The result then follows from the next two lemmas. 
\end{proof}

\begin{lemma} \label{lem-mfI-1-conv-p1}
Under the assumptions of Lemma \ref{lem-mfI-1-conv}, 
for $\ep>0$, as $\delta \to 0$, 
\begin{align} \label{eqn-Y-diff-p1}
 \limsup_N  \sup_{t\in [0,T]} \frac{1}{\delta}  \P \bigg(  \sup_{u \in [0,\delta]}\sup_{x \in [0,T']} \big|Y^N(t+u,x) - Y^N(t,x)\big| > \ep\bigg) \to 0, 
\end{align}
\end{lemma}

\begin{proof}
We have 
\begin{align*}
& \big|Y^N(t+u,x) - Y^N(t,x)\big|\\
 &= \Bigg| N^{-1} \sum_{i=A^N((t+u-x)^+)}^{A^N(t+u)}\big(  \bone_{\tau^N_i + \eta_i >t+u} -F^c(t+u-\tau^N_i) \big)
 - N^{-1} \sum_{i=A^N((t-x)^+)}^{A^N(t)}\big(  \bone_{\tau^N_i + \eta_i >t} -F^c(t-\tau^N_i) \big) \Bigg| \\
 &= \Bigg| N^{-1} \sum_{i=A^N((t-x)^+)}^{A^N(t+u)} \Big[\big(  \bone_{\tau^N_i + \eta_i >t+u} -F^c(t+u-\tau^N_i) \big) - \big(  \bone_{\tau^N_i + \eta_i >t} -F^c(t-\tau^N_i) \big) \Big]  \\
 & \quad - N^{-1} \sum_{i=A^N((t-x)^+)}^{A^N(t+u-x)^+ -1}\big(  \bone_{\tau^N_i + \eta_i >t+u} -F^c(t+u-\tau^N_i) \big)
 + N^{-1} \sum_{i=A^N(t)+1}^{A^N(t+u)}\big(  \bone_{\tau^N_i + \eta_i >t} -F^c(t-\tau^N_i) \big) \Bigg| \\
 & \le  N^{-1} \sum_{i=A^N((t-x)^+)}^{A^N(t+u)} \Big[ \bone_{t< \tau^N_i + \eta_i \le t+u} + \big(F^c(t-\tau^N_i) - F^c(t+u-\tau^N_i)  \big) \Big]  \\
 & \quad + \Bigg| N^{-1} \sum_{i=A^N((t-x)^+)}^{A^N(t+u-x)^+ -1}\big(  \bone_{\tau^N_i + \eta_i >t+u} -F^c(t+u-\tau^N_i) \big) \Bigg|
 + \Bigg| N^{-1} \sum_{i=A^N(t)+1}^{A^N(t+u)}\big(  \bone_{\tau^N_i + \eta_i >t} -F^c(t-\tau^N_i) \big) \Bigg| \\
 & \le  N^{-1} \sum_{i=A^N((t-x)^+)}^{A^N(t+u)}  \bone_{t< \tau^N_i + \eta_i \le t+u}  + N^{-1} \sum_{i=A^N((t-x)^+)}^{A^N(t+u)}  \big(F^c(t-\tau^N_i) - F^c(t+u-\tau^N_i)  \big) \\
 & \quad + \Big( \bar{A}^N((t+u-x)^+)- \bar{A}^N((t-x)^+) \Big) + \Big( \bar{A}^N(t+u) - \bar{A}^N(t) \Big). 
\end{align*}

Then we obtain
\begin{align} \label{eqn-Y-diff-p1-1}
 & \P \left(  \sup_{u \in [0,\delta]}\sup_{x \in [0,T']} \big|Y^N(t+u,x) - Y^N(t,x)\big| > \ep\right)  \non \\
 & \le  \P \left( N^{-1} \sum_{i=A^N((t-T')^+)}^{A^N(t+\delta)}  \bone_{t <\tau^N_i + \eta_i \le t+\delta}  > \ep/3 \right)  \non \\
 & \quad +  \P \left( N^{-1} \sum_{i=A^N((t-T')^+)}^{A^N(t+\delta)}\big( F^c(t-\tau^N_i)  -F^c(t+\delta-\tau^N_i) \big)  > \ep/3 \right)  \non \\
&  \quad +2 \P \left(\sup_{0\le s\le T} \big|\bar{A}^N(s+\delta) - \bar{A}^N(s) \big|  > \ep/6 \right)\,.
\end{align}
Let $\breve{Q}(ds,dr,dz)$ denote a PRM on $\RR_+^3$ with mean measure $\nu(ds,dr,dz) = ds dr F(dz)$ and $\wt{Q}$ denote the associated compensated PRM. 
By the Markov inequality, we obtain the first term is bounded by $9\ep^{-2}$ times
\begin{align*}
& \E  \left[ \left(  N^{-1} \sum_{i=A^N((t-T')^+)}^{A^N(t+\delta)}  \bone_{t <\tau^N_i + \eta_i \le t+\delta} \right)^2 \right] \\
&=  \E  \left[ \left(  N^{-1}  \int_{(t-T')^+}^{t+\delta} \int_0^\infty \int_{t-s}^{t+\delta-s} \bone_{r \le {\Upsilon}^N(s^-)} \breve{Q} (ds,dr, dz)  \right)^2 \right]  \\
& \le 2 \E  \left[ \left(  N^{-1}  \int_{(t-T')^+}^{t+\delta} \int_0^\infty \int_{t-s}^{t+\delta-s} \bone_{r \le 
N \bar{\Upsilon}^N(s^-)} \wt{Q} (ds,dr, dz)  \right)^2 \right] \\
& \quad + 2 \E  \left[ \left(  \int_{(t-T')^+}^{t+\delta}   \big( F(t+\delta-s) - F(t-s) \big) \bar{\Upsilon}^N(s) ds  \right)^2 \right] \\
& =  2 N^{-1}\E  \left[  \int_{(t-T')^+}^{t+\delta} \big( F(t+\delta-s) - F(t-s) \big) \bar{\Upsilon}^N(s) ds \right] \\
& \quad +  2 \E  \left[ \left(  \int_{(t-T')^+}^{t+\delta}   \big( F(t+\delta-s) - F(t-s) \big) \bar{\Upsilon}^N(s) ds \right)^2 \right] \\
& \le 2 \lambda^* N^{-1} \int_{(t-T')^+}^{t+\delta} \big( F(t+\delta-s) - F(t-s) \big) ds \\
& \quad +  2 \left(  \lambda^* \int_{(t-T')^+}^{t+\delta}   \big( F(t+\delta-s) - F(t-s) \big) ds \right)^2,
\end{align*}
where the last inequality follows from \eqref{eqn-barPhi-bound}. 
The first term converges to zero as $N\to \infty$, and we note that
\[ \int_{(t-T')^+}^{t+\delta}   \big( F(t+\delta-s) - F(t-s) \big) ds=\int_{t-(t-T')^+}^{t-(t-T')^++\delta}F(r)dr-\int_0^\delta F(r)dr\le\delta\,,\]
since $F(r)\le1$. Hence
\begin{align}\label{eqn-Y-diff-p1-2}
\frac{1}{\delta}\left(  \lambda^* \int_{(t-T')^+}^{t+\delta}   \big( F(t+\delta-s) - F(t-s) \big) ds \right)^2\le\lambda^\ast \delta \to 0, \qasq \delta \to 0. 
\end{align}
For the second term in \eqref{eqn-Y-diff-p1-1}, we have
\begin{align*}
 &  \E\left[  \left( 
 \int_{(t-T')^+}^{t+\delta}  \big( F^c(t-s)  -F^c(t+\delta-s) \big)d \bar{A}^N(s)  \right)^2 \right]  \\
 & \le 2 \E\left[  \left( 
 \int_{(t-T')^+}^{t+\delta}  \big( F^c(t-s)  -F^c(t+\delta-s) \big)d \bar{M}_A^N(s)  \right)^2 \right]  \\ 
 & \quad +  2\E\left[  \left( 
 \int_{(t-T')^+}^{t+\delta}  \big( F^c(t-s)  -F^c(t+\delta-s) \big) \bar{\Upsilon}^N(s)ds  \right)^2 \right], 
\end{align*}
where the first term converges to zero as $N \to \infty$ by the convergence $\bar{M}^N_A (t) \to 0$ in probability, locally uniformly in $t$, 
while the second term is bounded as in \eqref{eqn-Y-diff-p1-2}.

For the last term  in \eqref{eqn-Y-diff-p1-1}, we use the martingale decomposition of $\bar{A}^N$ and the bound for $\bar{\Upsilon}^N$ in  \eqref{eqn-barPhi-bound}, and obtain
\begin{align}\label{eqn-barAn-inc-bound}
\sup_{0\le t\le T}\big|\bar{A}^N(t+\delta) - \bar{A}^N(t)\big| \le 2\sup_{0\le t\le T+\delta}\big|\bar{M}_A^N(t)\big|  + \lambda^* \delta,
\end{align}
which, since $\bar{M}_A^N(t)\to0$ locally uniformly in $t$, implies that, provided $\delta<\ep/\lambda^\ast$,
\begin{align*}
\limsup_{N\to\infty}\P\left(\sup_{0\le t\le T}\big|\bar{A}^N(t+\delta) - \bar{A}^N(t)\big|\ge\ep\right)=0\,.
\end{align*}
Thus we have shown that \eqref{eqn-Y-diff-p1} holds. 
\end{proof}

\begin{lemma} \label{lem-mfI-1-conv-p2}
Under the assumptions of Lemma \ref{lem-mfI-1-conv}, 
for $\ep>0$, as $\delta \to 0$, 
\begin{align}\label{eqn-Y-diff-p2}
& \limsup_N  \sup_{x\in [0,T']} \frac{1}{\delta}  \P \bigg(  \sup_{v \in [0,\delta]}\sup_{t \in [0,T]} \big|Y^N(t,x+v) - Y^N(t,x)\big| > \ep\bigg) \to 0.
\end{align}
\end{lemma}

\begin{proof}
Observe that 
\begin{align*}
 Y^N(t,x+v) - Y^N(t,x)
 = N^{-1} \sum_{i=A^N((t-x-v)^+)}^{A^N((t-x)^+)}\big(  \bone_{\tau^N_i + \eta_i >t} -F^c(t-\tau^N_i) \big),
\end{align*}
from which we obtain
\begin{align*}
&  \P \bigg(  \sup_{v \in [0,\delta]}\sup_{t \in [0,T]} \big|Y^N(t,x+v) - Y^N(t,x)\big| > \ep\bigg)  \\
& \le \P \bigg(  \sup_{v \in [0,\delta]}\sup_{t \in [0,T]} \big| \bar{A}^N((t-x)^+) - \bar{A}^N((t-x-v)^+) \big| > \ep\bigg)  \\
&\le  \P \bigg( \sup_{t \in [0,T]} \big|\bar{A}^N((t-x)^+) - \bar{A}^N((t-x-\delta)^+) \big| > \ep\bigg)\,. 
\end{align*}
Then the claim follows from the same argument as the one used to treat the last term in \eqref{eqn-Y-diff-p1-1} in the end of the proof of the previous lemma. 
\end{proof}

\subsection{Convergence of the aggregate infectivity process}

 Recall $ \mathcal{I}^N$ in \eqref{eqn-cI-n}, and let  $\overline{\mathcal{I}}^N := N^{-1} \mathcal{I}^N$. 
  Define 
  \begin{align} \label{eqn-wt-In}
 \wt{\mathcal{I}}^N(t) 
 &:= N^{-1} \sum_{j=1}^{I^N(0)}   \bar\lambda (\tilde{\tau}^N_{j,0}+t)   
+ N^{-1} \sum_{i=1}^{A^N(t)}  \bar\lambda(t-\tau^N_i), \quad t \ge 0.
 \end{align}

 \begin{lemma} \label{lem-mcI-diff0}
 Under Assumptions \ref{AS-FLLN-Initial} and \ref{AS-lambda},  along a convergent subsequence of $\bar{A}^N$ which converges weakly to $\bar{A}$,  we have
in probability, 
 $$
 \overline{\mathcal{I}}^N - \widetilde{\mathcal{I}}^N  \to 0 \qinq D \qasq N \to \infty.
  $$
 \end{lemma}
 
 \begin{proof}
We write 
 $$
 \overline{\mathcal{I}}^N(t) - \widetilde{\mathcal{I}}^N (t) = \overline\Xi^N_0(t) + \overline\Xi^N_1(t), 
 $$
 where 
  \begin{align*} 
 \overline\Xi^N_0(t) &= N^{-1} \sum_{j=1}^{I^N(0)} \big( \lambda_j^0 (\tilde{\tau}^N_{j,0}+t) -  \bar\lambda (\tilde{\tau}^N_{j,0}+t) \big) \,\,, \\
\overline\Xi^N_1(t) 
& =   N^{-1} \sum_{i=1}^{A^N(t)} \big( \lambda_i(t-\tau^N_i) - \bar\lambda(t-\tau^N_i) \big)   \,\,. 
 \end{align*}
 We first consider $\overline\Xi^N_0(t)$.   
 For each fixed $t$, by conditioning on $\sigma\{I^N(0,y): 0\le y \le \bar{x}\} = \sigma\{\tilde{\tau}^N_{j,0}, j=1,\dots, I^N(0)\}$,  we obtain 
 \begin{align*}
 \E \big[\big( \overline\Xi^N_0(t) \big)^2\big] 
 &= N^{-1} \E\bigg[  \int_0^{\bar{x}} v(y+t) d \bar\mfI^N(0,y) \bigg] \to 0 \qasq N \to \infty.
 \end{align*}
We then have for $t,u>0$,
\begin{align*}
\big| \overline\Xi^N_0(t+u)  -  \overline\Xi^N_0(t) \big|
 &\le N^{-1} \sum_{j=1}^{I^N(0)} \Big| \lambda_j^0 (\tilde{\tau}^N_{j,0}+t+u) - \lambda_j^0 (\tilde{\tau}^N_{j,0}+t)\Big| \\
 & \quad + N^{-1} \sum_{j=1}^{I^N(0)} \Big| \bar\lambda (\tilde{\tau}^N_{j,0}+t+u) - \bar\lambda (\tilde{\tau}^N_{j,0}+t)\Big| \\
& =:\Delta_0^{N,1}(t,u) + \Delta_0^{N,2}(t,u)\,. 
\end{align*}
Then by Assumption \ref{AS-lambda}, writing $\lambda^0_j(t) = \sum_{\ell=1}^k\lambda^{0,\ell}_j(t) \bone_{[\zeta_j^{\ell-1}, \zeta_j^\ell)}(t)$,  we have
\begin{align}\label{eqn-Delta-0N1-p1}
\Delta_0^{N,1}(t,u) &= N^{-1} \sum_{j=1}^{I^N(0)} \bigg| \sum_{\ell=1}^k \lambda_j^{0,\ell} (\tilde{\tau}^N_{j,0}+t+u) \bone_{[\zeta_j^{\ell-1}, \zeta_j^\ell)}(\tilde{\tau}^N_{j,0}+t+u)  \non \\
& \qquad \qquad \qquad - \sum_{\ell=1}^k \lambda_j^{0,\ell} (\tilde{\tau}^N_{j,0}+t) \bone_{[\zeta_j^{\ell-1}, \zeta_j^\ell)}(\tilde{\tau}^N_{j,0}+t) \bigg| \non \\
&\le  N^{-1} \sum_{j=1}^{I^N(0)}  \sum_{\ell=1}^k |\lambda_j^{0,\ell} (\tilde{\tau}^N_{j,0}+t+u)  -  \lambda_j^{0,\ell} (\tilde{\tau}^N_{j,0}+t) |  \bone_{\zeta_j^{\ell-1} \le   \tilde{\tau}^N_{j,0}+t \le \tilde{\tau}^N_{j,0}+t+u \le \zeta_j^\ell} \non\\
& \quad + \lambda^* N^{-1} \sum_{j=1}^{I^N(0)}  \sum_{\ell=1}^k \bone_{   \tilde{\tau}^N_{j,0}+t \le \zeta_j^\ell \le \tilde{\tau}^N_{j,0}+t+u } \non \\
& \le \varphi(u) \bar{I}^N(0) + \lambda^* N^{-1}  \sum_{\ell=1}^k \sum_{j=1}^{I^N(0)}  \bone_{   \tilde{\tau}^N_{j,0}+t \le \zeta_j^\ell \le \tilde{\tau}^N_{j,0}+t+u } \,\,.
\end{align}
Both terms on the right hand side are increasing in $u$, and thus, we have
\begin{align*}
\sup_{0 \le u \le \delta} \Delta_0^{N,1}(t,u)  \le  \varphi(\delta) \bar{I}^N(0) + \lambda^* N^{-1}  \sum_{\ell=1}^k \sum_{j=1}^{I^N(0)}  \bone_{   \tilde{\tau}^N_{j,0}+t \le \zeta_j^\ell \le \tilde{\tau}^N_{j,0}+t+\delta } \,\,.
\end{align*}
Here for the second term, we have
\begin{align*}
& N^{-1}  \sum_{\ell=1}^k \sum_{j=1}^{I^N(0)}  \bone_{\tilde{\tau}^N_{j,0}+t \le \zeta_j^\ell \le \tilde{\tau}^N_{j,0}+t+\delta } \\
&=N^{-1}  \sum_{\ell=1}^k \sum_{j=1}^{I^N(0)}
\bigg[\bone_{\tilde{\tau}^N_{j,0}+t \le \zeta_j^\ell \le \tilde{\tau}^N_{j,0}+t+\delta }
-\big(F_\ell( \tilde{\tau}^N_{j,0}+t+\delta) - F_\ell( \tilde{\tau}^N_{j,0}+t)\big) \big) \bigg]\\
&\quad+N^{-1}  \sum_{\ell=1}^k \sum_{j=1}^{I^N(0)}\big(F_\ell( \tilde{\tau}^N_{j,0}+t+\delta) - F_\ell( \tilde{\tau}^N_{j,0}+t)\big) \big),
\end{align*}
hence
\begin{align}\label{decomp+}
\P&\left(N^{-1}  \sum_{\ell=1}^k \sum_{j=1}^{I^N(0)}  \bone_{\tilde{\tau}^N_{j,0}+t \le \zeta_j^\ell \le \tilde{\tau}^N_{j,0}+t+\delta }>\ep\right)\non\\
&\le\sum_{\ell=1}^k\P\left(N^{-1}   \sum_{j=1}^{I^N(0)}
\bigg[\bone_{\tilde{\tau}^N_{j,0}+t \le \zeta_j^\ell \le \tilde{\tau}^N_{j,0}+t+\delta }
-\big(F_\ell( \tilde{\tau}^N_{j,0}+t+\delta) - F_\ell( \tilde{\tau}^N_{j,0}+t)\big) \big) \bigg]>\ep/2k\right)\non\\
&\quad+ \sum_{\ell=1}^k\P\left(N^{-1}  \sum_{j=1}^{I^N(0)}\big(F_\ell( \tilde{\tau}^N_{j,0}+t+\delta) - F_\ell( \tilde{\tau}^N_{j,0}+t)\big) \big)>\ep/2k\right)\,.
\end{align}
The first term on the right of \eqref{decomp+} tends to $0$ as $N\to\infty$, since
by conditioning on $\sigma\{\mfI^N(0,y): 0\le y \le \bar{x}\} = \sigma\{\tilde{\tau}^N_{j,0}, j=1,\dots, I^N(0)\}$, and since the $\zeta_j^\ell$'s are mutually independent and globally independent of the $\tilde{\tau}^N_{j,0}$'s,  we obtain 
\begin{align*}
&  \E \bigg[ \bigg( N^{-1}  \sum_{j=1}^{I^N(0)}  \big( \bone_{   \tilde{\tau}^N_{j,0}+t \le \zeta_j^\ell \le \tilde{\tau}^N_{j,0}+t+\delta } - \big(F_\ell( \tilde{\tau}^N_{j,0}+t+\delta) - F_\ell( \tilde{\tau}^N_{j,0}+t)\big) \big) \bigg)^2  \bigg]  \\
& = \E \bigg[ N^{-1} \int_0^{\bar{x}} \big(F_\ell(y+t+\delta) - F_\ell( y+t)\big) \big)  \big[1- \big(F_\ell(y+t+\delta) - F_\ell( y+t)\big) \big) \big]  \bar\mfI^N(0,dy)  \bigg]\,.
\end{align*}
The second term on the right of \eqref{decomp+} equals
\[ \sum_{\ell=1}^k\P\left( \int_0^{\bar{x}}   \big(F_\ell( y+t+\delta) - F_\ell(y+t)\big) \big)   \bar\mfI^N(0,dy)>\ep/2k\right),\]
whose limsup as $N\to\infty$ is bounded from above by
\[ \sum_{\ell=1}^k{\bf1}\left\{ \int_0^{\bar{x}}   \big(F_\ell( y+t+\delta) - F_\ell(y+t)\big) \big)   \bar\mfI(0,dy)\ge\ep/2k\right\}\,.\]
Since for each $1\le \ell\le k$, 
\[ \delta\mapsto\int_0^{\bar{x}}   \big(F_\ell( y+t+\delta) - F_\ell(y+t)\big) \big)   \bar\mfI(0,dy)\]
is continuous and equals $0$ at $\delta=0$, for any $\ep>0$, there exists $\delta>0$ small enough such that
the above quantity vanishes.
Thus, we have shown that
\begin{align} \label{eqn-Delta-0N1-p3}
\limsup_{N\to\infty}\sup_{t\in [0,T]}  \frac{1}{\delta}\P \left( \sup_{0 \le u \le \delta} \Delta_0^{N,1}(t,u) > \ep/3\right)  \to 0, \qasq \delta \to 0. 
\end{align}

Next, consider $\Delta_0^{N,2}(t,u)$, which is $\Delta_0^{N,1}(t,u)$, with the $j$--th term in the absolute value being replaced by its conditional expectation given $\tilde{\tau}^N_{j,0}$. The computations which led above to 
\eqref{eqn-Delta-0N1-p1} give
\begin{align*}
\sup_{0 \le u \le \delta} \Delta_0^{N,2}(t,u)  \le  \varphi(\delta) \bar{I}^N(0) + \lambda^* N^{-1}  \sum_{\ell=1}^k \sum_{j=1}^{I^N(0)}  \big(F_\ell( \tilde{\tau}^N_{j,0}+t+\delta) - F_\ell( \tilde{\tau}^N_{j,0}+t)\big) \,\,.
\end{align*}
So the same arguments as those used above yield that \eqref{eqn-Delta-0N1-p3} holds with 
$\Delta_0^{N,1}(t,u)$ replaced by $\Delta_0^{N,2}(t,u)$.

Thus we have shown that in probability,
$
\overline\Xi^N_0 \to 0$ in $D$ as $ N \to \infty$.
The convergence $\overline\Xi^N_1\to 0$ in $D$ in probability follows from the proof of Lemma 4.6 in \cite{FPP2020b}. In fact, the above proof of $
\overline\Xi^N_0 \to 0$ can be adapted to that proof by observing the similar roles of $A^N$ and $\mfI^N(0,\cdot)$.
 This completes the proof. 
\end{proof}

 \begin{lemma} \label{lem-sI-conv}
 Under Assumptions \ref{AS-FLLN-Initial} and \ref{AS-lambda},  along a convergent subsequence of $\bar{A}^N$ which converges weakly to $\bar{A}$, 
 \begin{equation}
 \overline{\sI}^N \RA \wt{\sI} \qinq D \qasq N \to \infty,
 \end{equation}
where $ \wt{\sI}(t)$ is given by
\begin{align}\label{limI}
 \wt{\sI}(t) 
  &= \int_0^{\bar{x}} \bar\lambda(y+t)  \bar{\mfI}(0,dy)  + \int_0^t  \bar\lambda(t-s)   d\bar{A}(s)\,, \quad t\ge 0. 
\end{align}
 \end{lemma}

\begin{proof}
By the above lemma, it suffices to show that 
 \begin{equation}
 \wt{\sI}^N \RA \wt{\sI} \qinq D \qasq N \to \infty. 
 \end{equation}
The expression of $ \wt{\sI}^N$ in  \eqref{eqn-wt-In} can be rewritten as 
  \begin{align} \label{eqn-wt-In-2}
 \wt{\mathcal{I}}^N(t) 
&= \int_0^{\bar{x}} \bar\lambda(y+t) \bar{\mfI}^N(0,dy) + \int_0^t \bar\lambda(t-s) d\bar{A}^N(s). 
 \end{align}
  It follows from Lemma \ref{le:Portmanteau} that for any $t>0$, as $N\to\infty$,  $\wt{\sI}^N(t) \RA \wt{\sI}(t)$.
 It remains to show that the sequence $\wt{\sI}^N$ is tight in $D$. For that purpose, exploiting the Corollary on page 83 of  \cite{billingsley1999convergence}, it suffices to show that for any 
  $\ep>0$, 
  \begin{align}
   \lim_{\delta\to0}\limsup_N\frac{1}{\delta} \P\left(\sup_{0\le u\le \delta}\left|\int_0^{\bar{x}} \bar\lambda(y+t+u) \bar{\mfI}^N(0,dy) -\int_0^{\bar{x}} \bar\lambda(y+t) \bar{\mfI}^N(0,dy) \right|>\ep\right)=0,\label{estimaccroist1}
   \\
   \lim_{\delta\to0}\limsup_N\frac{1}{\delta} \P\left(\sup_{0\le u\le \delta}
   \left|  \int_0^{t+u} \bar\lambda(t+u-s) d\bar{A}^N(s)- \int_0^t \bar\lambda(t-s) d\bar{A}^N(s)    \right|>\ep\right)=0\,. \label{estimaccroist2}
   \end{align}
  
\eqref{estimaccroist1} follows from the fact that, with $G_\delta(s):=\sup_{0\le u\le \delta}|\bar\lambda(s+u)-\bar\lambda(s)|$,
\begin{align*}
\limsup_N\P\left(\sup_{0\le u\le \delta}\left|\int_0^{\bar{x}} (\bar\lambda(y+t+u)-\bar\lambda(y+t)) \bar{\mfI}^N(0,dy)   \right|>\ep\right) \le{\bf1}\left\{\int_0^{\bar{x}}G_\delta(y+t)\bar{\mfI}(0,dy)>\ep\right\}\,.
\end{align*}
 Now $\bar{\mfI}(0,dy)$ a.e., $G_\delta(y+t)\to0$, and since $0\le G_\delta(y+t)\le\lambda^\ast$, it follows from Lebesgue's dominated convergence that $\int_0^{\bar{x}}G_\delta(y+t)\bar{\mfI}(0,dy)\to0$, as $\delta\to0$,hence for $\delta>0$ small enough, this quantity is less than $\ep$, and the indicator vanishes.
 
 It remains to establish \eqref{estimaccroist2}. We have
 \begin{align*}
 \int_0^{t+u} \bar\lambda(t+u-s) d\bar{A}^N(s)- \int_0^t \bar\lambda(t-s) d\bar{A}^N(s) &=
 \int_t^{t+u} \bar\lambda(t+u-s) d\bar{A}^N(s)\\&\quad+\int_0^t[\bar\lambda(t+u-s) -\bar\lambda(t-s) ]d\bar{A}^N(s),
 \end{align*}
 hence
  \begin{align*}
& \sup_{0\le u\le\delta}\left|\int_0^{t+u} \bar\lambda(t+u-s) d\bar{A}^N(s)- \int_0^t \bar\lambda(t-s) d\bar{A}^N(s)\right| \\
 &\le(\lambda^\ast)^2\delta+\lambda^\ast\int_0^tG_\delta(t-s)ds +\lambda^\ast
 \big|\bar{M}_A^N(t+\delta)-\bar{M}_A^N(t)\big|+\left|
 \int_0^tG_\delta(t-s)d\bar{M}^N_A(s)\right|
 \end{align*}
 The result follows since the sum of the two first terms on the right are less than $\ep/2$ for $\delta>0$ small enough, while the two last terms tend to $0$, as $N\to\infty$.
 \end{proof}
 
 \subsection{Completing the proof of Theorem \ref{thm-FLLN}}
 
By Lemmas \ref{lem-barIN0-conv} and \ref{lem-mfI-1-conv}, we have that, along a subsequence,
\begin{align*}
\bar\mfI^N(t,x) = \bar\mfI^N_0(t,x) + \bar\mfI^N_1(t,x) \Rightarrow \wt{\mfI}(t,x) = \bar\mfI_0(t,x) + \bar\mfI_1(t,x) \in D_D \qasq N \to \infty,
\end{align*}
where $\bar\mfI_0(t,x)$ and $\bar\mfI_1(t,x)$ are given in \eqref{eqn-bar-mfI-0-rep1} and \eqref{eqn-bar-mfI-1-rep1}, respectively. Also recall that $\bar{S}^N=\bar{S}^N(0) - \bar{A}^N$ by \eqref{eqn-Sn-rep}.
We need to show the joint convergence
$$
(\bar{S}^N, \bar\mfI^N, \overline{\sI}^N) \Rightarrow (\bar{S}, \wt{\mfI}, \wt{\sI}) \qinq D\times D_D\times D \qasq N \to \infty.
$$
or equivalently, 
\begin{equation}\label{jointconv}
(\bar{A}^N, \bar\mfI^N, \overline{\sI}^N) \Rightarrow  (\bar{A}, \wt{\mfI}, \wt{\sI}) \qinq D\times D_D \times D \qasq N \to \infty.
\end{equation}
 Indeed, first thanks to Lemma \ref{lem-mcI-diff0}, we can replace $\overline{\sI}^N$ by 
$\wt{\sI}^N$. 
Next we have the decompositions 
\begin{align*}
\bar\mfI^N&=\bar\mfI_0^N+\bar\mfI_1^N,\\
\wt{\sI}^N&=\wt{\sI}_0^N+\wt{\sI}_1^N,
\end{align*}
where $\wt{\sI}_0^N$ and $\wt{\sI}_1^N$ are respectively the first and the second term
on the right of the identity \eqref{eqn-wt-In-2}. 
By the independence of the quantities associated with initially and newly infected individuals, it suffices to prove the joint convergence of the processes $(\bar\mfI_0^N, \wt{\sI}_0^N)$
and that of the processes $(\bar{A}^N,\bar\mfI_1^N,\wt{\sI}_1^N)$ separately. 
We have proved in Lemma \ref{lem-barIN0-conv} that $\bar\mfI_0^N\to\bar\mfI_0$ in $D_D$ in probability, and it follows from the arguments in the proof of Lemma \ref{lem-sI-conv} that 
$\wt{\sI}_0^N\to\wt{\sI}_0$ in $D$ in probability, where $\wt{\sI}_0$ is the first term on the right of the identity \eqref{limI}. Hence, the joint convergence $(\bar\mfI_0^N, \wt{\sI}_0^N) \to(\bar\mfI_0,\wt{\sI}_0)$ in $D^2$ in probability is immediate.

Exploiting again \eqref{eqn-Yn-conv-0},  we see that 
 the joint convergence  $(\bar{A}^N,\bar\mfI_1^N,\wt{\sI}_1^N)\Rightarrow  (\bar{A}, \bar\mfI_1, \wt{\sI}_1)$ 
 will be a consequence of
\begin{equation}\label{jointconv*}
(\bar{A}^N, \breve\mfI^N_1, \wt{\sI}_1^N) \Rightarrow  (\bar{A}, \bar\mfI_1, \wt{\sI}_1) \qinq D\times D_D \times D \qasq N \to \infty.
\end{equation}
where  $\wt{\sI}_1$ denotes the second term on the right of the identity  \eqref{limI}.
Since $\breve{\mfI}^N_1(t,x)  = \int_{(t-x)^+}^t F^c(t-s) d \bar{A}^N(s)$ and 
$\wt{\sI}_1^N(t)=\int_0^t \bar\lambda(t-s) d\bar{A}^N(s)$,
the joint finite dimensional convergence is a consequence of the continuous mapping theorem and 
Lemma \ref{le:Portmanteau}. Hence the result follows from tightness. We have proved the joint convergence property in 
\eqref{jointconv}.

Recall the expression of $\overline\Upsilon^N(t) =  \bar{S}^N(t)\overline{\mathcal{I}} ^N(t)$. 
Applying the continuous mapping theorem again, we obtain 
  that
  $$
\overline{\Upsilon}^N(t) \Rightarrow \bar\Upsilon(t) = \bar{S}(t) \wt{\sI}(t)  \qinq D \qasq N \to \infty. 
$$
Thus by \eqref{eqn-An-conv1}, 
we conclude that 
$$
\bar{A}^N \Rightarrow \bar{A} = \int_0^{\cdot} \bar\Upsilon(s)ds =  \int_0^{\cdot}  \bar{S}(s) \wt{\sI}(s) ds \qinq D \qasq N \to \infty. 
$$
Therefore, the limit $(\bar{S}, \wt{\sI})$ satisfies the set of integral equations in \eqref{eqn-barS}, \eqref{eqn-overline-cal-I-2}
and the limit $\wt{\sI}$ coincides with $\overline{\sI}$ defined by \eqref{eqn-overline-cal-I-2}. Then, the limit $\wt{\mfI}$ coincides with $\bar{\mfI}$ in \eqref{eqn-barI}. 
The limits $\bar{I}$ in \eqref{eqn-barIt} and $\bar{R}$ in \eqref{eqn-barR} then follow immediately. 
The set of integral equations has a unique deterministic solution. Indeed, it is easy to see that the system of equations \eqref{eqn-barS} and \eqref{eqn-overline-cal-I-2} (together with the first part of \eqref{eqn-bar-Upsilon}) has a unique solution $(\bar{S}, \overline{\sI})$, given the initial values  $\bar\mfI(0,\cdot)$.
The other processes $\bar\mfI, \bar{I}, \bar{R}$ are then uniquely determined.  
Hence the whole sequence converges in probability.

From \eqref{eqn-barI}, 
we deduce that for all $t>0$, 
\begin{align} 
\bar{\mfI}_x(t,0) &= \lim_{x\to 0} \frac{\bar\mfI(t,x) - \bar\mfI(t,0)}{x} = \lim_{x\to 0} \frac{\bar\mfI(t,x)}{x}  = \bar\Upsilon(t). \non 
\end{align}
This prove the second equality in \eqref{eqn-bar-Upsilon}. 

It remains to prove the continuity. The continuity in $t$ of $\bar{S}(t)$ is clear. Let us prove that $t\mapsto\overline\sI(t)$ is continuous. Since $\lambda_i$ is c\`adl\`ag and bounded, it is easily checked that
$t\mapsto\bar\lambda(t)=\E[\lambda(t)]$ is also c\`adl\`ag. In fact it is continuous if all the $F_\ell$'s for $1\le \ell\le k$ are continuous. The points of discontinuity of $\bar\lambda(t)$ are the points where one of the laws of the $\zeta^\ell$ has some mass. The set of those points is at most countable. Consequently, if $t_n\to t$, the set of $y$'s where $\bar\lambda(t_n+y)$ may not converge to $\bar\lambda(t+y)$ is at most countable, and this is a set of zero
$\bar\mfI(0,dy)$ measure. Since moreover $0\le\bar\lambda(t_n+y)\le\lambda^\ast$, $t\to\int_0^{\bar{x}}\bar\lambda(y+t)\bar\mfI(0,dy)$ is continuous. Let us now consider the second term in \eqref{eqn-overline-cal-I-2}.
We first note that since $\bar\lambda(t-s)\le\lambda^\ast$ and $\bar{S}(t)\le1$, it follows from \eqref{eqn-overline-cal-I-2}, \eqref{eqn-bar-Upsilon} and Gronwall's Lemma that $\overline\sI(t)\le\lambda^\ast e^{\lambda^\ast t}$. 
Let $t_n\to t$. We have
\begin{align*}
& \left|\int_0^t\bar\lambda(t-s)\bar\Upsilon(s)ds-\int_0^{t_n}\bar\lambda(t_n-s)\bar\Upsilon(s)ds\right|
\\
& \le \int_0^t|\bar\lambda(t-s)-\bar\lambda(t_n-s)|\Upsilon(s)ds+(\lambda^\ast)^2e^{\lambda^\ast (t\vee t_n)}|t-t_n|.\end{align*}
Clearly the above right hand side tends to $0$, as $n\to\infty$. A similar argument shows that $\bar{R}$ and $\bar{I}$ are continuous, and that $(t,x)\mapsto\bar\mfI(t,x)$ is continuous. Finally, since the convergence holds in 
$D\times D\times D_D\times D$ and the limits are continuous, the convergence is locally uniform in $t$ and $x$. This completes the proof of Theorem \ref{thm-FLLN}.

\section{Appendix: Proof of Theorem \ref{thm-DD-conv0}}\label{sec:App}

Given $\delta>0$, we define the two sets 
\begin{align*}
\Gamma_{T,\delta}&:=\Big\{0,\delta,2\delta,\ldots,\lfloor\frac{T}{\delta}\rfloor\delta\Big\},\\
\Gamma_{S,\delta}&:=\Big\{0,\delta,2\delta,\ldots,\lfloor\frac{S}{\delta}\rfloor\delta \Big\}\,.
\end{align*}
For any $t\in[0,T]$, we define $\gamma_{T,\delta}(t)$ to be the element of $\Gamma_{T,\delta}$ such that
$\gamma_{T,\delta}(t)\le t<\gamma_{T,\delta}(t)+\delta$, and for any $s\in[0,S]$, we define $\gamma_{S,\delta}(s)$ to be the element of $\Gamma_{S,\delta}$ such that
$\gamma_{S,\delta}(s)\le s<\gamma_{S,\delta}(s)+\delta$.

Let $(t,s)$ and $(t',s')$ be two points in $[0,T]\times[0,S]$ such that $|t-t'|\vee|s-s'|\le\delta$. We have
\begin{align*}
X^N(t,s)-X^N(t',s')&=X^N(t,s)-X^N(t,\gamma_{S,\delta}(s))+X^N(t,\gamma_{S,\delta}(s))-
X^N(\gamma_{T,\delta}(t),\gamma_{S,\delta}(s))\\
&\quad+X^N(\gamma_{T,\delta}(t),\gamma_{S,\delta}(s))
-X^N(\gamma_{T,\delta}(t'),\gamma_{S,\delta}(s))\\&\quad+X^N(\gamma_{T,\delta}(t'),\gamma_{S,\delta}(s))
-X^N(\gamma_{T,\delta}(t'),\gamma_{S,\delta}(s'))\\
&\quad+X^N(\gamma_{T,\delta}(t'),\gamma_{S,\delta}(s'))
-X^N(t',\gamma_{S,\delta}(s'))+X^N(t',\gamma_{S,\delta}(s'))-X^N(t',s'). 
\end{align*}
Hence
\begin{align*}
\P&\left(\sup_{0\le t,t'\le T; 0\le s,s'\le S;|t-t'|\vee|s-s'|\le\delta}|X^N(t,s)-X^N(t',s')|>\epsilon \right)\\&\le 
3\sum_{s\in\Gamma_{S,\delta}}\P\left(\sup_{0\le t\le T, u\in[0,\delta]}|X^N(t,s+u)-X^N(t,s)|>\epsilon/6\right)\\&\quad+3\sum_{t\in\Gamma_{T,\delta}}\P\left(\sup_{0\le s\le S, u\in[0,\delta]}|X^N(t+u,s)-X^N(t,s)|>\epsilon/6\right)\\
&\le 3\left(\frac{1}{\delta}+1\right)\sup_{0\le s\le S}\P\left(\sup_{0\le t\le T, u\in[0,\delta]}|X^N(t,s+u)-X^N(t,s)|>\epsilon/6\right)\\
&\quad+3\left(\frac{1}{\delta}+1\right)\sup_{0\le t\le T}\P\left(\sup_{0\le s\le T, u\in[0,\delta]}|X^N(t+u,s)-X^N(t,s)|>\epsilon/6\right). 
\end{align*}
It then follows from (ii) that, as $\delta\to0$,
\[ \limsup_N \P\left(\sup_{0\le t,t'\le T; 0\le s,s'\le S;|t-t'|\vee|s-s'|\le\delta}|X^N(t,s)-X^N(t',s')|>\epsilon\right)\to0.\]
This, combined with (i), implies the result. \hfill $\Box$

\section*{Acknowledgement}
We thank the reviewers on the helpful comments that have improved the exposition of our paper. Guodong Pang is partly supported by the NSF grant DMS-2216765. 

\bibliographystyle{abbrv}
\bibliography{Epidemic-Age-PDE}

\end{document}